\documentclass[10pt,reqno]{amsart}

 \usepackage{amsmath,amsfonts,amssymb,amscd,amsthm,amsbsy,bbm, epsf,calc,enumerate,color,graphicx,verbatim,cite,import}
\usepackage{color}
\usepackage{datetime,appendix}
\usepackage{chapterbib,epstopdf}

\usepackage{float}

\usepackage{ifpdf}
\ifpdf
    \usepackage[colorlinks,citecolor=black,linkcolor=black]{hyperref}
\fi

\usepackage[mathlines]{lineno}

\theoremstyle{plain}
\newtheorem{definition}{Definition}[section]
\newtheorem{thm}{Theorem}[section]
\newtheorem{theorem}{Theorem}[section]
\newtheorem{mainthm}{Theorem}

\newtheorem{lemma}[thm]{Lemma}
\newtheorem{cor}[thm]{Corollary}

\newtheorem{proposition}[thm]{Proposition}

\theoremstyle{definition}

\theoremstyle{remark}                  

\newtheorem{remark}[thm]{Remark}

\definecolor{darkgreen}{rgb}{0,0.4,0}

\def \l {\left(}
\def\r {\right)}

\numberwithin{equation}{section}

\def\tr{\textnormal{tr}}

\def\XXint#1#2#3{{\setbox0=\hbox{$#1{#2#3}{\int}$ }
\vcenter{\hbox{$#2#3$ }}\kern-.6\wd0}}

\title[On Continuity Equations in Space-Time Domains]{On Continuity Equations in Space-Time Domains}
 
\author[Y Zhang]{\bfseries Yuming Zhang}

\address{
Department of Mathematics \\ 
University of California   \\ 
Los Angeles\\
USA}
\email{yzhangpaul@math.ucla.edu}

\begin{document}

\vspace{18mm} \setcounter{page}{1} \thispagestyle{empty}

\begin{abstract}

In this paper we consider a class of continuity equations that are conditioned to stay in general space-time domains, which is formulated as a continuum limit of interacting particle systems.
Firstly, we study the well-posedness of the solutions and provide examples illustrating that the stability of solutions is strongly related to the decay of initial data at infinity. In the second part, we consider the vanishing viscosity approximation of the system, given with the co-normal boundary data. If the domain is spatially convex, the limit coincides with the solution of our original system, giving another interpretation to the equation.
\end{abstract}

\maketitle

\section{Introduction}

Let $\Omega_T$ be a given space-time domain in $\mathbb{R}^d\times (0,T)$, denoted by
 $$
 \Omega_T:=\cup_{0<t<T}\left(\Omega(t)\times\{t\}\right).
 $$ In this domain we consider a continuity equation of the form:
\begin{equation*}
\left\{\begin{aligned}&
\frac{\partial}{\partial t}\mu(x,t)+\nabla\cdot (v \mu)(x,t)= 0 \quad\text{ in }\Omega_T, \\
&\text{ where }v(x,t)=-(\epsilon \nabla \mu/\mu+\nabla V +\nabla W *\mu)(x,t) \hbox{ with }  \epsilon\geq 0,\\
&\mu(x,0)=\mu_0(x).
\end{aligned}
\right.
\end{equation*}
in the space of probability measures, with the constraint that the support of  $\mu$ lies in the closure of $\Omega_T$.  When $\epsilon>0$, this constraint yields the co-normal boundary data on the lateral boundary of $\Omega_T$ \eqref{eqn2}. The first-order system, $\epsilon=0$, will be formulated using a projection operator \eqref{eqn1}: we will show that this system can be obtained as the vanishing viscosity limit as $\epsilon\rightarrow 0$. 

\medskip



The above system describes the density of moving particles which are confined to some region and flow with a velocity field ${v}$ inside of the domain. One part of the velocity field is generated from interactions between different particles represented by the interaction potential $W$, given by $$(\nabla W *\mu) (x,t):=\int_{\overline{\Omega(t)}}\nabla W(x-y)d \mu(y,t).$$ This type of problem arises in many applications with various interaction kernel $W$, such as in swarming models with $ W(x)=-Ce^{-|x|},W(x)=-Ce^{-|x|^2}$ and in models of chemotaxis with $W(x)=\frac{1}{2\pi}\log|x|$, see \cite{prox,6} for more references. At the same time, the particles are subject to an external potential $V(x)$. Both $V, W$ are assumed to be smooth and $\lambda$-convex. More assumptions will be presented in section \ref{sub2.1} and \ref{sub3.1}. For the diffusion term, the model takes into account random movements of the particles.

\medskip 
  
In the first part of this paper, we consider $\epsilon=0$. 
Let $c(x,t)$ be the speed of the boundary (with positive sign if the boundary is expanding) and $n(x,t)$ be the unit space outer normal for $x\in\partial\Omega(t)$. We set $c(x,t)=0,n(x,t)=0$ if $(x,t)$ are not on the boundary. For simplicity we may omit the dependencies and write $c,n$. For each $(x,t)\in \overline{\Omega_T}$, we define a projection operator $ P_{x,t}: \mathbb{R}^d\rightarrow \mathbb{R}^d$ as follows
\begin{equation}\label{Pt}
 P_{x,t} (v)=\begin{cases}v &\text{ if }v\cdot n\leq c,\\
v-(v\cdot n)n+cn &\text{ if }v\cdot n> c.
\end{cases}
\end{equation}
Note at $x\in\Omega(t)$ in the interior, $ P_{x,t} $ is an identity map on $v$.
We refer readers to \cite{prox, wuli} where they defined a similar projection operator on stationary domains. 

\begin{figure}[t]
\centering\includegraphics[width=0.4\textwidth]{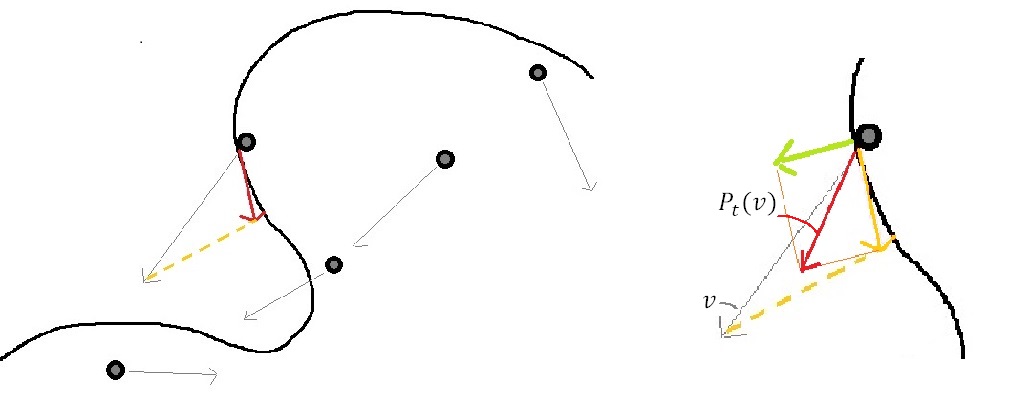}
\caption{the particle system and the $ P_{x,t} $ operator}\end{figure}

Set $\overline{\Omega}_T:=\cup_{0<t<T}(\overline{\Omega(t)}\times\{t\})$. We formulate the equation as:
\begin{equation}\label{eqn1}
\left\{\begin{aligned}&
\frac{\partial}{\partial t}\mu(x,t)+\nabla\cdot \left(\mu  P_{x,t} (-\nabla V -\nabla W *\mu)\right)(x,t)  = 0  &\text{ in }&\overline{\Omega}_T, \\
&\mu(x,0)=\mu_0(x)  &\text{ on }&  \overline{\Omega(0)}.
\end{aligned}
\right.
\end{equation}
Here $\mu_0$ is a probability measure on $\overline{\Omega(0)}$. First we assume it is compactly supported, and later we will also consider probability measures with exponential decay properties. In both cases $\mu_0$ has finite second moment.

Our main contributions are two-fold. The first part of the results are mainly motivated by the previous work by Carrillo, Slepcev and Wu \cite{prox}, where they show the well-posedness of equation \eqref{eqn1} in stationary, non-convex domains with compactly supported initial data. We generalize the well-posedness result to general space-time domains and allow non-compactly supported initial data. Second, we show that \eqref{eqn1} can be obtained as the limit as $\epsilon\rightarrow 0$ of the diffusion equation \eqref{eqn2} given with the co-normal boundary data, imposing the additional condition that the domain is bounded and spatially convex. This result is significant since it provides a natural justification for the first-order system \eqref{eqn1}.




\subsection{First Main Result}


As in \cite{prox}, we use particle approximations. 
The hard part is to show the limit of particle approximating solutions is indeed a weak solution to \eqref{eqn1} due to the fact that the projecting operator $ P_{x,t} (v)$ is discontinuous with respect to $x,t$ on the boundary. So instead we show that the limit is a gradient flow solution by taking limit of the ``curve of maximal slope" (see inequality \eqref{extra}) which is then a weak solution.

The novelty in this paper is that, comparing with \cite{prox}, for space-time domains an extra term ($\tilde{E}$ below or see \eqref{extra}) appears in the curve of maximal slope, 
\[{\mathcal{E}}(\mu):=
\int^t_s\int_{ \mathbb{R}^d }(\underbrace{\left(w- P_{x,r} w\right)\cdot P_{x,r} w}_{\tilde{E}} +\frac{1}{2}\left|P_{x,r} w\right|^2 ) (x,r)d\mu dr\]
where $w(\mu)(x,r):=-(\nabla V+\nabla W*\mu)(x,r)$. Intuitively this extra term $\tilde{E}$ comes from the moving boundary constraints and the possible situation that particles attempt to move out of the domain with potential velocity $w$ but end up moving with the boundary with velocity $ P_{x,t} w$. According to the definition of $ P_{x,t} $, it is not hard to see that $\tilde{{E}}(\mu)$ is only nontrivial if $\mu$ is singular with mass concentrating on the boundary. Alternatively if the domain is stationary, this term vanishes, since if the boundary speed $c(x,t)=0$, $ P_{x,t} $ at boundary points are projections onto the tangential plane of $(x,t)$.

We need more careful analysis because of this term. To be specific, the key point is to show the lower semi-continuity of $\mathcal{E}(\mu)$ in $\mu$. This can be proved if $\tilde{E}(w)+\frac{1}{2}|P_{x,t}w|^2 $ is convex in $w$ which is the case when we have expanding boundary. However it is not convex if the boundary speed is negative. This difficulty can be resolved by observing the two facts: the sequence $w^n$ we take limit of converges uniformly in $(x,t)$ if $W$ is $C^2$; $E(w)(x,t)$ is lower semi-continuous in $x,t$. These two also guarantee the lower semi-continuity of $\mathcal{E}$ in $\mu$ (see Lemma \ref{lemu}).

\medskip

The energy associated to \eqref{eqn1} contains potential energy and interaction energy:
\begin{equation}\label{phi}
\phi(\mu):=\int_{ \mathbb{R}^d }V(x)d\mu(x,t)+\frac{1}{2}\int_{\mathbb{R}^{2d}}W(x-y) d\mu(y,t)d \mu(x,t).
\end{equation}
Here $\mu(\cdot,t)$ will only be supported in $\overline{\Omega(t)}$.

For non-compactly supported initial data with exponential decay property (see condition \textsc{(R)} in section \ref{sub2.4}), existence of solutions can be done by particle approximation as well as truncation method. 
 Uniqueness of solutions satisfying exponential decay is proved by a stability estimate \eqref{stab_p}.
We will provide examples in Theorem \ref{example} showing that the requirement below that $p<1$, as well as the exponential decay condition, are essential.

\medskip

Now we summarize the main theorem in part one. We use $d_W(\cdot,\cdot)$ to denote the 2-Wasserstein distance between probability measures. The Wasserstein metric and the notion of weak solutions will be discussed in section \ref{1.1}.


\begin{mainthm}  \label{A}
Assume conditions \textsc{(O1)(C1)-(C3)} hold (see section \ref{sub2.1} for details). Let $\mu_0$ be a probability measure supported in $\overline{\Omega(0)}$ and fix any $T>0$. 
\begin{flushleft}
(a) (Theorem \ref{thmu}) Suppose $\mu_0$ is compactly supported. Then there is a weak solution $\mu(\cdot)$ to equation \eqref{eqn1} and $\mu(t)$ is compactly supported for $0\leq t\leq T$. If $\mu^1(\cdot), \mu^2(\cdot)$ are two solutions with compact support in $[0,T]$, then there exists $C$ such that
\[d_W\left(\mu^1(t),\mu^2(t)\right)\leq C d_W(\mu^1_0,\mu^2_0)\text{ for all } 0\leq t\leq T.\]

(b) (Theorem \ref{converge})
Suppose $\mu_0$ satisfies the exponentially decay property \textsc{(R)}. Then there exists a weak solution $\mu(\cdot)$ of equation \eqref{eqn1} and $\mu(t)$ satisfies \textsc{(R)} for $0\leq t\leq T$. If $\mu^1(\cdot), \mu^2(\cdot)$ are two solutions with initial data $\mu^1_0,\mu^2_0$ and $\mu^1(\cdot), \mu^2(\cdot)$ satisfy \textsc{(R)} in $[0,T]$, then for any $0<p<1$, if $t$ and $d_W(\mu^1_0,\mu^2_0)$ are small enough, we have
\begin{equation}\label{stab_p}
d_W\left(\mu^1(t),\mu^2(t)\right)\leq 2d_W(\mu^1_0,\mu^2_0)^p .
\end{equation}

(c) (Theorem \ref{example}) For non-convex unbounded domains, examples can be found that the $``p"$ above cannot be improved to $1$. Furthermore for $\mu_0$ with less decay, even the above stability property does not hold.
\end{flushleft}

\end{mainthm}


\subsection{Second Main Result}

In the second part, we consider the case $\epsilon>0$:
\begin{equation}\label{eqn2}
\left\{\begin{aligned}&
\frac{\partial}{\partial t}\mu^\epsilon-\nabla\cdot ( \epsilon\nabla \mu^\epsilon+\nabla V \mu^\epsilon+(\nabla W *\mu^\epsilon)\mu^\epsilon )= 0 & \text{ in }&\Omega_T,\\
&\left(\epsilon\nabla \mu^\epsilon+\nabla V \mu^\epsilon+(\nabla W *\mu^\epsilon)\mu^\epsilon + c\mu^\epsilon\right)\cdot n=0 &
\text{ on }& \partial_l\Omega_T,\\
&\mu^\epsilon(x, 0)=\mu_0(x) & \text{ on }&  \Omega(0).
\end{aligned}
\right.
\end{equation}
Here $\partial_l\Omega_T:=\overline{\Omega}_T\backslash{\Omega}_T$ is the lateral boundary of $\Omega_T$; $c$ is the speed of the boundary.
The co-normal boundary condition above gives the mass preservation. The associated energy $\phi^\epsilon$ is given by 
\begin{equation} \label{phid}
    \begin{aligned}
\phi^\epsilon(\mu)&=\mathcal{U}^\epsilon(\mu)+\mathcal{V}(\mu)+\mathcal{W}(\mu)\\
&:=\epsilon\int_{\mathbb{R}^d}(u\log u)(x,t) dx+\int_{\mathbb{R}^d}V(x)d\mu(x,t)+\frac{1}{2}\int_{\mathbb{R}^{2d}}W(x-y) d\mu(y,t)d \mu(x,t)  .
\end{aligned}
\end{equation}
In the first term, $u$ is the probability density of $\mu$ if $\mu$ is absolutely continuous with respect to Euclidean measure. We set $\phi^\epsilon(\mu)=\infty$ if $\mu$ is not absolutely continuous.

\medskip

With the convergence of $\epsilon\rightarrow 0$ in mind, we show the existence of solutions by discrete-time gradient-flow (JKO) solutions (see \cite{JKO}). For this purpose, technically we further require $V,W$ to be $C^2$ and bounded below. In time-dependent domain, the scheme is slightly different from the standard version that we minimize each movement among probability measures with support contained in $\Omega_T$ (see \eqref{JKOscheme}).
To obtain the continuum time limit of the discrete-time solutions, we show the uniform boundedness of the second moment and the boundedness of $\phi^\epsilon$ along solutions from the discrete scheme. This is one part that the analysis for problems on stationary domains that cannot be directly carried over for the time-dependent domains. The problem is solved in Proposition \ref{prop 2}, the proof of which is inspired by the work of Di Marino, Maury and Santambrogio \cite{sweeping} who encountered the same problem. Also let us mention that solutions obtained in this way inherit the gradient flow structure which will be important later.

By a gradient flow argument, we have uniqueness of solutions in bounded and spatially convex moving domains, see Remark \ref{remarkunique}. For non-convex bounded stationary domain, we give a uniqueness proof based on an $L^2$ stability estimate, see Theorem \ref{uniqueness}.

\medskip

After establishing the well-posedness of weak solutions, we send $\epsilon\rightarrow 0$. 
It will be proved that if the domain is bounded and spatially convex, equations \eqref{eqn2} are indeed the vanishing viscosity approximation of the first order equation \eqref{eqn1} in the first part. This convergence justifies the formulation of equation \eqref{eqn1}, in addition to the derivation via particle system. 

We use a Gronwall type argument. By the gradient flow theory (mainly Sections 8,10,11 \cite{gflow}), the time derivative of the 2-Wasserstein distance between $\mu$ and $\mu^\epsilon$ is related to the Fr\'echet subdifferentials of their energy $\phi,\phi^\epsilon$ at $\mu,\mu^\epsilon$ respectively. We want to use the convexity of the energies to finish the argument and to do so we also need to consider $\phi^\epsilon(\mu)$. A serious problem arises that the value of $\mathcal{U}^\epsilon$ at $\mu$ can be infinity. This is because, in general even with smooth initial data, $\mu$ can concentrate mass in finite time as discussed in \cite{carrillo}. 

To overcome this problem, we develop a new modification method. We select a $\tilde{\mu}=\tilde{u}dx$ which is close to $\mu$, and $\tilde{u}$ is bounded point-wisely by $\epsilon^{-\alpha}$ for some $0<\alpha<1$. Using that the domain is bounded, we obtain $\mathcal{U}^\epsilon(\tilde{\mu})\rightarrow 0$ as $\epsilon\rightarrow 0$. Then by the variational inequality \eqref{phiexi}, it turns out that we need $\mu,\tilde{\mu}$ to be close not only in 2-Wasserstein metric but also in Pseudo-Wasserstein distance with base $\mu^\epsilon$ (the definition will be given in section \ref{1.1}). Finding such a modification of $\mu$ is one of the most technical part of this paper since the information we have about the base $\mu^\epsilon$ at each time is limited. Simple convolution with a bump function won't give the expected $\tilde{\mu}$, see Appendix \ref{counterexample}. The modification is done for general absolutely continuous base measure in Lemma \ref{modi}. We need the convexity of the domain and we will use generalized geodesics in probability measure space with Pseudo-Wasserstein metric and Brunn-Minkowski inequality.

\medskip

Now let us give the main theorem of the second part of this paper.

\begin{mainthm}\label{mainthmB}
Assume conditions \textsc{(C1)-(C4)(O1)(O2)} hold (see sections \ref{sub2.1}, \ref{sub3.1} for details), $\mu_0$ is absolutely continuous (with respect to Lebesgue measure) probability measures supported in $\Omega(0)$ with finite second moments, and $\phi^\epsilon(\mu_0)<\infty$ for some $\epsilon>0$. Then for any fixed $T>0$
\begin{flushleft}
(a) (Theorem \ref{existence})
There exists a weak solution $\mu^\epsilon(\cdot)$ to equation \eqref{eqn2} and for each $t\geq 0$, $\mu^\epsilon(t)$ is absolutely continuous with respect to Lebesgue measure.\\


(b) (Theorem \ref{convergence})
Suppose $\Omega(t)$ is bounded and convex for all $t\in [0,T]$. Let $\mu^\epsilon(\cdot)$ be the weak solution to equation \eqref{eqn2} and $\mu(\cdot)$ be the weak solution to equation \eqref{eqn1} with the same initial data. Then there exist constants $c,C$ that 
\[d^2_W(\mu^\epsilon(t),\mu(t)) \leq C\epsilon^\frac{1}{d+2}te^{ct}  \text{ for all }t\in[0,T].\]
\end{flushleft}
\end{mainthm}

\medskip

Lastly let us mention that in \cite{cozzi2016aggregation}, the vanishing viscosity limit problem in the whole domain was studied in the case when $V = 0$ and $-W$ is the Newtonian potential. Their proof heavily relies on the specific choice of kernel $W$, and also the fact that the domain is $\mathbb{R}^d$ which eliminates the task of determining the limiting boundary condition. 

\leavevmode
\medskip

\textbf{Acknowledgements.}
The author would like to thank his advisor Inwon Kim for suggesting the problem, which was motivated from a conversation with Jos\'e Carrillo, as well as for the stimulating guidance and discussions. The author would like to thank Katy Craig and Wilfrid Gangbo for fruitful discussions and thank Alp\'ar Rich\'ard M\'esz\'aros for the helpful conversations including discussing Brunn-Minkowski theorem. Also, reference \cite{cozzi2016aggregation} was pointed out by Katy Craig.
The author would like to thank the referee for a careful reading of the manuscript and for her/his constructive comments.

\section{Notations and Preliminaries}\label{1.1}

Suppose $\cup_t \partial\Omega(t)$ is smooth.
For $x\in\partial\Omega(t)$, we write $n(x,t)$ (or simply $n$) as the unit spatial outer normal vector and $c(x,t)$ (or $c$) as the speed of the boundary at $(x,t)$. They are defined such that, if letting
\[ y_s\in argmin_{y_s\in\partial\Omega(t+s)}|y_s-x|,\]
then
\[c(x,t)n(x,t)=\lim_{s\to 0}\frac{y_s-x}{s}.\]

Throughout the paper, we fix a time $T>0$ which is assumed to be large. We say a constant is universal if it only depends on $T$ and constants in conditions \textsc{(O1)(O2)(C1)-(C4)} ($\lambda,r_p, L$ and bounds about $c, V, W$). We denote by $C$ a constant which may depend on universal constants and $\mu_0$, possibly changing from one estimate to another. 

A spatial ball in $\mathbb{R}^d$ centered at $x$ with radius $r$ is denoted by $B_r(x)$, and we may simply write $B_r$ if $x$ is the origin. Given $S\subseteq \mathbb{R}^d$, we use the notation $vol\{S\}$ as the Lebesgue measure of $S$.


Given a probability measure $\mu$, we write $m_2(\mu)=\int_{\mathbb{R}^d} |x|^2 d\mu$ as the second moment of $\mu$.
The set of all probability 
measures on $\overline{\Omega}$ with 
finite second moment is denoted by $\mathcal{P}_2(\overline{\Omega})$. The set of absolutely continuous (with respect to Lebesgue measure) probability measures with finite second moment is written as $\mathcal{P}_2^a$. For $\mu\in\mathcal{P}_2^a$, we usually write $\mu=u\mathcal{L}^d$ where $u$ is its density.
For probability measures supported in $\overline{\Omega}$, we will think of them as measures in $ \mathbb{R}^d $, extended by $0$ outside $\overline{\Omega}$. 


\medskip

Now we give the definition of weak (measure) solutions to equations \eqref{eqn1} and \eqref{eqn2}.
\begin{definition}
Assume $\mu_0\in \mathcal{P}_2(\overline{\Omega(0)})$. A locally absolutely continuous (in Wasserstein metric) curve $\mu(\cdot)\in\mathcal{P}_2(\mathbb{R}^d)$ is a weak solution to \eqref{eqn1} for $t\in [0,T]$ if
\[v= P_{x,t} (-\nabla V-\nabla W*\mu) \in L^1_{loc}([0,T); L^2(\mu(t)))\]
and for all $\varphi\in C_c^\infty\left( \mathbb{R}^d \times (0,T)\right) $:
\[\int_{ \mathbb{R}^d \times (0,T)}\varphi_t d\mu dt+\int_{ \mathbb{R}^d \times (0,T)}(\nabla\varphi\cdot v) d\mu dt=0,\quad \mu(0)=\mu_0,\]
and for all $t\in [0,T]$, $supp(\mu(t))\subseteq \overline{\Omega(t)}.$
\end{definition}

\begin{definition}
Assume $\mu_0\in \mathcal{P}^a_2({\Omega(0)})$. A locally absolutely continuous curve $\mu(\cdot)\in\mathcal{P}^a_2(\mathbb{R}^d)$ is a weak solution to \eqref{eqn2} for $t\in[0,T]$ if the density $u(\cdot)$ of $\mu(\cdot)$ satisfies
\[\nabla u \in L^{1}(\Omega_T)\]
and
for all $\varphi\in C_c^\infty\left( \mathbb{R}^d \times (0,T)\right) $:
\[\int_{ \mathbb{R}^d \times (0,T)}\left(\varphi_t-\nabla V\cdot\nabla\varphi-(\nabla W*\mu)\cdot \nabla\varphi\right) d\mu dt-\int_{ \Omega_T} \epsilon{\nabla u}\cdot\nabla\varphi\, dxdt =0,\quad \mu(0)=\mu_0,\]
and for all $t\in[0,T]$, 
$supp(\mu(t))\subseteq {\Omega(t)}.$
\end{definition}

\medskip

Now we discuss the Wasserstein metric and we refer readers to \cite{gflow} for details.
Suppose $X,Y$ are measurable subsets of $ \mathbb{R}^d $ and $\mu_1\in \mathcal{P}_2(X), \mu_2\in \mathcal{P}_2(Y)$. A plan between $\mu_1,\mu_2$ is any Borel measure $\gamma$ on ${ X\times Y }$ which has $\mu_1$ as its first marginal and $\mu_2$ as its second marginal. We write $\gamma\in\Gamma(\mu_1,\mu_2)$. It has been shown that there exists an optimal transport plan $\gamma\in \Gamma(\mu_1,\mu_2)$ such that
\[\int_{ X \times  Y }|x-y|^2d\gamma(x,y)=\min\left\{\int_{ X \times  Y }|x-y|^2d\gamma'(x,y),\gamma'\in\Gamma(\mu_1,\mu_2)\right\}. \]
The above quantity is defined to be the \textit{2-Wasserstein distance} between $\mu_1,\mu_2$ (the Kantorovich's formulation). Throughout this paper we use this distance for probability measures with notation $d_W(\cdot,\cdot)$ unless otherwise stated. And later by Wasserstein distance (metric) we mean 2-Wasserstein distance (metric).
We denote the set of optimal transport plans between $\mu_1$ and $\mu_2$ by $\Gamma_0(\mu_1,\mu_2)$.

Let $\mu_2\in \mathcal{P}_2(Y)$, a measurable function $\textbf{t}: Y\rightarrow X$ transports $\mu_2$ onto $\mu_1\in\mathcal{P}_2(X)$ if $\mu_1(B)=\mu_2(\textbf{t}^{-1}(B))$ for all measurable $B\subseteq X$, and we write $\mu_1=\textbf{t}_\# \mu_2$.
If $\mu_2\in \mathcal{P}^a_2(Y)$, then for any $\mu_1\in \mathcal{P}_2(X)$ there is an optimal transport map $\textbf{t}_{\mu_2}^{\mu_1}: Y\rightarrow X$ such that ${\textbf{t}_{\mu_2}^{\mu_1}}_\# \mu_2=\mu_1$ (With reference to \cite{mccann1995}). And we have, in Monge's formulation,
\[d_W^2(\mu_1,\mu_2)=\int_Y |\textbf{t}_{\mu_2}^{\mu_1}(x)-x|^2d \mu_2(x).\]

Given ${\mu_1},{\mu_2}\in\mathcal{P}_2(X),\mu\in\mathcal{P}^a_2(X)$. Let $\textbf{t}_\mu^{\mu_1},\textbf{t}_\mu^{\mu_2}$ be an optimal transport maps from $\mu$ to $\mu_1$ and $\mu_2$
respectively. Then the \textit{Pseudo-Wasserstein distance} with base $\mu$ is defined as
\[d_\mu^2(\mu_1,\mu_2)=\int_{X}|\textbf{t}_\mu^{\mu_1}-\textbf{t}_\mu^{\mu_2}|^2d\mu.\]
By Proposition 1.15 \cite{katy}, $d_\mu$ is a metric on \[\mathcal{P}_\mu(X):=\left\{\mu' \in \text{probability measures on } X, d_W(\mu,\mu')<+\infty \right\}.\]
And we have for any $\mu$,
$d_W(\cdot,\cdot)\leq d_\mu(\cdot,\cdot).$

\medskip

Finally let us recall the Brunn-Minkowski theorem.

\begin{lemma}\label{minkowski}
Let $d \geq 1$ and let $A$ and $B$ be two nonempty compact subsets of $\mathbb{R}^d$. Then the following inequality holds:
\[ vol\{A+B\}^{1/d}\geq vol\{ A\}^{1/d}+vol\{ B\}^{1/d},\]
where
$A+B$ denotes the Minkowski sum:
\[
A+B:=\{\,a+b\in \mathbb {R}^{d} \,|\, a\in A,\, b\in B\,\}.\]
\end{lemma}

\section{Part One. Nonlocal First Order Equations}

\subsection{Settings and Assumptions}\label{sub2.1}

We study equation \eqref{eqn1} in the first part of this paper. 

Suppose $S\subset\mathbb{R}^d$ is open and the boundary is $C^1$. Then we say $S$ is \textit{$r$-prox-regular} if for any point $x \in \partial S,y\in S$ we have
\[\langle n(x), y-x\rangle \leq \frac{|y-x|^2}{r}\]
where $n(x)$ is the unit normal at $x$ (see \cite{proxregular} for more results).
This is the same as: for any boundary point $x$, there is a ball of radius $r$ that intersects $\overline{S}$ at exactly $x$. 

Recall a $C^1$ function $f(x)$ is called \textit{$\lambda-$convex} in $S \subseteq  \mathbb{R}^d $ if
\[\langle \nabla f(x)-\nabla f(y),x-y \rangle \geq \lambda |x-y|^2 \text{  for all }x,y\in S.\]
Now we list the assumptions below.

\begin{flushleft}
\textbf{(O1)} For each $t$, $\Omega(t)$ is a non-empty open subset of $ \mathbb{R}^d $ which is always \textit{$r_p$-prox-regular} for some $r_p>0$. The lateral boundary $\cup_t\left(\partial\Omega(t)\times\{t\}\right)$ is $C^1$ in both time and space direction, particularly the boundary speed $c(x,t)$ is continuous if restricted to the boundary. In addition, we require
\begin{equation*}|c(x,t)|\leq C (1+|x|).
\end{equation*}
\end{flushleft}

\begin{flushleft}
\textbf{(C1)} $V\in C^1(\mathbb{R}^d)$ and $|\nabla V(x)|\leq C(1+|x|)$
for all $x\in \mathbb{R}^d$.\\
\textbf{(C2)} $W\in C^1(\mathbb{R}^d)$ with $|\nabla W(x)|\leq C(1+|x|)$ 
for all $x\in \mathbb{R}^d$, and $W(x)=W(-x)$. If \[\{(x,t),c(x,t)<0\}\ne \emptyset,\]
we require $W\in C^2(\mathbb{R}^d)$.\\
\textbf{(C3)} There exists some $\lambda\in\mathbb{R}$ such that $V,W$ are $\lambda$-convex in $\mathbb{R}^d$. 

\end{flushleft}

\subsection{Particle Approximations}

As stated in the introduction, we use particle approximations. Consider: $\mu_0=\sum_{i=1}^{N} m_i \delta_{x_{i,0}}$ where $N$ is a large integer and $ \sum_{i=1}^N m_i=1 $. We look for a solution of the form $\mu(t)=\sum_{i=1}^{N} m_i \delta_{x_i(t)}$. By the weak formulation, equation \eqref{eqn1} becomes
\begin{equation}\label{ODE}
\left\{\begin{aligned}
&\dot{x_i}(t)=P_{x,t}\left(w(\mu)(x_i(t),t)\right)  \text{ a.e. for } t>0\\
&x_i(0)=x_{i,0} \in \overline{\Omega(0)}
\end{aligned}
\right.
\end{equation}
where 
\begin{equation}\label{Fxmu}
w(\mu)(y,t):=-\nabla V(y)-\sum_j m_j \nabla W \left(y-x_j(t)\right).\end{equation}

The ODE can be solved by a differential inclusions argument; the proof is in the appendix.

\begin{proposition}\label{2.1}
Assume conditions \textsc{(O1)(C1)(C2)} hold. Let $\mu_0=\sum_{i=1}^N m_i\delta_{x_{i,0}}$ be the finite sum of delta masses. Then the ODE system \eqref{ODE}
has a locally absolutely continuous solution (for each $i$, $x_i(\cdot)$ is absolutely continuous).
\end{proposition}

\subsubsection{Some Estimates for Discrete Systems}

Since $\frac{d}{d t}x_j(t)= P_{x,t} \left(w(\mu)(x_j(t),t)\right)$ a.e. for $t>0$, we have
\[
|\dot{x_j}(t)|\leq |w(x_j(t),t)|+|c(x_j(t),t)|.
\]
By direct computations, \textsc{(C1),(C2)} and the fact that $\sum_j m_j=1$
\[|w(x_i(t),t)|\leq C\l 1+|x_i(t)|+\sum_j\l m_j|x_i(t)-x_j(t)|\r\r\leq C(1+|x_i(t)|+m_2(\mu(t))^\frac{1}{2}).\]
We know that $m_2(\mu_0)=\sum_i m_i x_{i,0}^2$ is bounded. Then
\begin{equation}\label{boundm2}
\begin{aligned}
\frac{d}{d t}\left(m_2(\mu(t)\right)&\leq 2\sum_i m_i |x_i(t)||\dot{x}_i(t)|\leq 2\sum_i m_i |x_i(t)|(|w(x_j(t),t)|+|c(x_i(t),t)|)\\
&\leq C \sum_i\left( m_i(|x_i(t)|+|x_i(t)|^2)+\sum_j m_i|x_i(t)|m_2(\mu(t))^\frac{1}{2} \right)
\leq C (1+ m_2(\mu(t)) ).
\end{aligned}
\end{equation}
This provides us a uniform bound for $m_2\left(\mu(t)\right)$ which only depends on $T, m_2(\mu_0)$. And then
\begin{equation}\label{bound F}
|w(x_i(t),t)| \leq C(1+|x_i(t)|+m_2\left(\mu(t)\right)^\frac{1}{2})\leq C(1+|x_i(t)|).\end{equation}
Also we have
\[\frac{d}{d t}|x_i(t)|^2\leq C|x_i(t)|(|w(x_i(t),t)|+|c(x_i(t),t)|)\leq C(1+|x_i(t)|^2).\]
This shows the linear growth of $x_i$ in time that
\begin{equation}\label{bound x2}
|x_i(t)|^2\leq C(1+|x_{i,0}|^2) \quad \text{ for }t\leq T.
\end{equation}
And this illustrates that the solutions $\mu(t)$ are always compactly supported in finite time. Also for particles starting outside $B_{CR}$, they will be outside $B_R$ for $t\in [0,T]$. This will be used in Theorem \ref{converge}.

Let $k(x,t) n=w(x,t)-P_{x,t}w(x,t)$ and so $k$ actually depends on $\mu$. Then
\begin{equation} \label{bound k}
|k(x_i(t),t)|\leq |w(x_i(t),t)|+|c(x(t),t)|\leq C(1+|x_i(t)|).\end{equation}
From the above \begin{equation}\label{wmi}
d_W^2(\mu(t),\mu(s))\leq \sum_i m_i|x_i(t)-x_i(s)|^2\leq C\sum_i m_i\int_s^t(1+|x_i(r)|)^2dr\leq C|t-s|
\end{equation}
which gives the $\frac{1}{2}$-H\"{o}lder continuity of $\mu(\cdot)$ in Wasserstein distance.

\medskip

Now let us recall the \textit{metric derivative} of an absolutely continuous curve in $\mathcal{P}_2(\mathbb{R}^d)$,
\[|\mu'|(t):=\lim_{h\rightarrow 0} \frac{d_W\left(\mu(t+h),\mu(t)\right)}{h}.\]
Then we can show the following proposition regarding the well-posedness of solutions for the projected discrete systems.

\begin{proposition}\label{thm1} Assume \textsc{(O1)(C1)(C2)} and $\mu_0=\sum_{i=1}^N m_i\delta_{x_{i,0}}\in\mathcal{P}(\overline{\Omega(0)})$, then equation \eqref{eqn1} has a weak solution $\mu(t)\in \mathcal{P}_2\left(\overline{\Omega(t)}\right)$ with $\mu(0)=\mu_0$. 
And $v(x,t):= P_{x,t} \left(-\nabla V(x)-\nabla W*\mu(x)\right)$ satisfies that \[ |\mu'|^2(t)\leq\int_{ \mathbb{R}^d }|v|^2 d\mu(x,t) \text{ for a.e. $t>0$}.\] \end{proposition}

\begin{proof}
Let $\zeta(x,t)\in C_c^\infty\left( \mathbb{R}^d \times (0,T)\right)$ be a test function.
By Proposition \ref{2.1},
\begin{align*}
    0&=\frac{d}{d t}\int_{ \mathbb{R}^d }\zeta(x,t)d\mu(t)
=\sum m_i \frac{\partial\zeta}{\partial t}(x_i(t),t)+m_i \nabla_x \zeta(x_i(t),t)\cdot P_{x,t}\left(w(x_i(t),t))\right)\\
&=\int_{ \mathbb{R}^d }\frac{\partial\zeta}{\partial t}(x,t)d\mu(t)+\int_{ \mathbb{R}^d }\nabla \zeta(x,t)\cdot  P_{x,t} \left(-\nabla V(x)-\nabla W*\mu (x)\right)d\mu.
\end{align*}
By \eqref{wmi}, $\mu(t)$ is an absolutely continuous curve with respect to Wasserstein metric. By the previous estimates we have $v(x,t) \in L^1([0,T); L^2(\mu))$. 
By Theorem 8.3.1 \cite{gflow}, 
\[
|\mu'|^2(t)\leq\int_{ \mathbb{R}^d }| P_{x,t} \left(-\nabla V(x)-\nabla W*\mu(x)\right)|^2 d\mu(x,t) \text{ for a.e. } t>0.\]

\end{proof}

\subsubsection{Stability of Discrete Solutions}

The following proposition gives the stability result of solutions in the discrete case. The proof is similar to the one in Proposition 5.1 \cite{prox}. The only difference is the movement of the boundary which can be controlled by condition \textsc{(O1)}.

\begin{proposition}\label{thm2}
Assume \textsc{(O1)(C1)-(C3)} hold. Suppose $\mu^1,\mu^2$ are solutions with discrete type initial measures $\mu^1_0,\mu^2_0$ as in Proposition \ref{thm1}. Then there exists a constant $C$ depending on the support of the initial data, the conditions and $T$, such that
\[d_W\left(\mu^1(t),\mu^2(t)\right)\leq Cd_W\left(\mu^1_0,\mu^2_0\right)  \quad\text{ for all }0<t\leq T.
\]
\end{proposition}

\begin{proof}
For $i=1,2$, let $\mu^i(t),v^i(x,t),w^i(\mu)(x,t) (=w(x_i,t))$ be defined as in Proposition \ref{2.1} and let $k^i(x,t)$ be such that $v^i(x,t)=w^i(x,t)-k^i(x,t) n(x,t)$. By $r_p$-prox regularity for each $y\in \overline{\Omega(t)}$
\[\langle k^i(x,t)n(x,t),y-x\rangle \leq \frac{|k^i(x,t)|}{r_p}|y-x|^2.\]
Estimate \eqref{bound x2} shows that for $t\in [0, T]$, $\mu^i(t)$ are compactly supported. By \eqref{bound k}, $\frac{|k^i(x,t)|}{r_p}\leq C$ in the support of $\mu^i(t)$.
Let $\gamma_t$ be an optimal transport plan between $\mu^1(t)$ and $\mu^2(t)$, then
\[-\int_{\mathbb{R}^{2d}}\langle k^1(x,t)n(x,t)-k^2(y,t)n(y,t),x-y\rangle d\gamma_t(x,y) \leq Cd_W^2\left(\mu^1(t),\mu^2(t)\right).
\]
By Theorem 8.4.7 from \cite{gflow}:
\begin{align*}
    &\frac{1}{2}\frac{d}{d t}d_W^2\left(\mu^1(t),\mu^2(t)\right)=\int_{\mathbb{R}^{2d}}\langle v^1(x,t)-v^2(y,t),x-y\rangle d\gamma_t(x,y)\\
\leq &\int_{\mathbb{R}^{2d}}\langle w^1(x,t)-w^2(y,t),x-y\rangle d\gamma_t(x,y)+Cd_W^2\left(\mu^1(t),\mu^2(t)\right).
\end{align*}
Since $w^i(x,t)=-(\nabla V+\nabla W*\mu^i )(x,t)$, by condition \textsc{(C3)},\[
-\int_{\mathbb{R}^{d}}\langle\nabla V(x)-\nabla V(y),x-y\rangle d\gamma_t(x,y)\leq -\lambda d_W^2(\mu^1(t),\mu^2(t)).\]
Since $W$ is $\lambda$-convex and even, we get 
\begin{align*}
&-\int_{\mathbb{R}^{2d}}\langle \nabla W*\mu^1(x)-\nabla W*\mu^2(y),x-y\rangle d\gamma_t(x,y)\\
=&-\frac{1}{2}\int_{\mathbb{R}^{2d}}\langle\nabla W(x-x')-\nabla W(y-y'),x-x'-y+y'\rangle d\gamma_t(x',y')d\gamma_t(x,y)\\
\leq& -\frac{\lambda}{2}\int_{\mathbb{R}^{2d}}|x-x'-y+y'|^2 d\gamma_t(x',y')d\gamma_t(x,y)\leq C(\lambda) d_W^2\left(\mu^1(t),\mu^2(t)\right).
\end{align*}
Here $C(\lambda)=\min\{0,2\lambda\}$. So in all
$\frac{d}{d t}d_W^2\left(\mu^1(t),\mu^2(t)\right)\leq C d_W^2\left(\mu^1(t),\mu^2(t)\right)$ which gives \[d_W^2\left(\mu^1(t),\mu^2(t)\right)\leq Cd_W^2\left(\mu^1(0),\mu^2(0)\right) \quad \text{ for all }0\leq t\leq T.\]
\end{proof}

\begin{remark}
If we were more careful on the dependence of the constant on $T$ and the support of the initial data (suppose $supp(\mu^i_0)\subset B_R$), we would find out that the constant in the above proposition can be bounded by $C \exp (CRT \exp(CT))$ where $C$ is a universal constant. 
\end{remark}

\subsection{Compactly Supported Initial Data}

Suppose $\mu_0\in\mathcal{P}_2(\overline{\Omega(0)})$ and consider a sequence of delta masses $\mu^n_0=\sum_{i=1}^{k(n)}m_i\delta_{x_i^n}$ converging to $\mu_0$ in Wasserstein metric. Without loss of generality, we assume that $\mu^n_0$ are supported in a compact set for all $n$.
Suppose $\mu^n(t)$ is a solution to \eqref{eqn1} given by Proposition \ref{thm1} with initial value $\mu^n_0$.
Proposition \ref{thm2} shows that for each $t$, $\mu^n(t)$ is a Cauchy sequence once $\mu^n_0$ is Cauchy. So we can write the limit as $\mu(t)$ which again is compactly supported.

Now we need to show that the limit $\mu(t)$ is indeed a solution to \eqref{eqn1}.
We do not expect showing that
\begin{equation}\label{vpf}v(x,t):= P_{x,t} (-\nabla V-\nabla W*\mu)(x,t)\end{equation} is the tangent velocity field of $\mu(t)$ by simply letting $n$ goes to infinity due to the discontinuity of $ P_{x,t} $ in $(x,t)$, which is also explained in Remark 3.2 \cite{prox}. To overcome this problem, we use gradient flow method.

Let us start with the following definition.

\begin{definition}
Let $\mu(\cdot)$ be an absolutely continuous curve in $\mathcal{P}_2(\mathbb{R}^d)$ with compact support in $\overline{\Omega_T}$.
We say that $\mu(\cdot)$ is \textit{a curve of maximal slope with respect to $\phi$ in a time-dependent domain}, if for all $0\leq s<t< T$:
\begin{align}\nonumber
&\phi\left(\mu(s)\right)\geq \phi\left(\mu(t)\right)+\frac{1}{2}\int^t_s|\mu'|^2(r)dr+\frac{1}{2}\int^t_s\int_{ \mathbb{R}^d }\left| P_{x,r} w(x,r))\right|^2d\mu(x,r)dr\\\label{extra}
&\quad +\int^t_s\int_{ \mathbb{R}^d }\left(w(x,r)- P_{x,r} w(x,r)\right) \cdot P_{x,r} w(x,r) d\mu(x,r)dr.
\end{align}
\end{definition}
Here as before, $w(x,r)=w(\mu)(x,r)=(-\nabla V-W*\mu)(x,r)$.
As mentioned in the introduction, the last term of \eqref{extra} appears because of the time-dependence of the domain. For $0\leq s<t\leq T$ we denote
\[\mathcal{E}(\mu):=\int^t_s\int_{ \mathbb{R}^d }w(x,r)\cdot P_{x,r} w(x,r)-\frac{1}{2}\left| P_{x,r} w(x,r)\right|^2d\mu(x,r)dr.\] 
We need the following lemma which shows the lower semi-continuity of $\mathcal{E}(\cdot)$ as $\mu^n\rightarrow \mu$. We require that $W\in C^2$ if the domain is compressing locally.

\begin{lemma}\label{lemu}
Notations are as above and suppose $\mu^n(t)$ converges to $\mu(t)$ in Wasserstein distance uniformly for all $t\in[0,T]$. If \textsc{(C1)(C2)} hold, then for all $0\leq s\leq t\leq T$ we have
\begin{equation}\label{lemu1}
\liminf_{n\rightarrow\infty}\mathcal{E}(\mu^n)\geq \mathcal{E}(\mu). 
\end{equation}
\end{lemma}
\begin{proof}
We consider the integrals on the boundary of $c(x,t)\geq 0$ and $c(x,t)<  0$ separately. Let us still write $t$ for time dependence and $w$ (or $w_i$) are continuous vector fields in $\overline{\Omega_T}$. For each $(x,t)\in\overline{\Omega_T}$, we define 
\[E(w):=w\cdot P_{x,t}w-\frac{1}{2}| P_{x,t} w|^2\] and we claim that the function $E(w)$ is convex in $w$ if $c(x,t)\geq 0$ which is equivalent to
\begin{equation}\label{wtleq}
E(w_s)\leq (1-s)E(w_0)+s E(w_1)
\end{equation}
where $w_s=(1-s)w_0+s w_1$. We denote $w_{i}^{N}:=w_i\cdot n$.
If $w_{0}^{N}\leq c,w_{1}^{N} \leq c$ or $w_{0}^{N}\geq c, w_{1}^{N} \geq c$, the proof is clear by definition. If $w_0^N\geq c, w_{1}^{N}\leq c$ and $w_{s}^{N} \leq c$, then  
\begin{align*}
&(1-s)E(w_0)+s E(w_1)-E(w_s)\\
   \geq&   (1-s)\left(w_0^N c-\frac{ c^2}{2}\right)+s\left(\frac{| w_1^N |^2}{2}\right)-\frac{| w_s^N|^2}{2}\\
=& \frac{1}{2} s(1-s)\l w_0^N-w_1^N \r^2-\frac{1}{2}(1-s)\l w^N_0-c\r^2. 
\end{align*}
Note $w_0^N-w_s^N=s \l w_0^N-w_1^N \r$. The above
\[\geq \frac{1-s}{2s} \l w_0^N-w_s^N \r^2-\frac{1-s}{2}\l w^N_0-c\r^2 \geq 0. \]
Similarly if $w_0^N\geq c, w_{1}^{N}\leq c$ and $w_{s}^{N} \geq c$, then 
\begin{align*}
    & (1-s)E(w_0)+s E(w_1)-E(w_s)\\
\geq &  (1-s)\left(c w_0^N-\frac{ c^2}{2}\right)+s\left(\frac{| w_1^N|^2}{2}\right)-\l cw_s^N -\frac{c^2}{2}\r\\
=& \frac{s}{2} \l w_1^N-c\r^2\geq 0. 
\end{align*}
So $E(w)$ is convex in $w$ if $c\geq 0$. For $(x,t)$ not on the boundary, $E(w)=\frac{w^2}{2}$. Notice \[E(w)-\frac{w^2}{2}\leq 0\]
and so $E(w)$ is lower semi-continuous in $x,t$ if $c\geq 0$. Then as did in Lemma 3.7 \cite{prox}, by Proposition 6.42 \cite{10} that for each $t$, there are two families of countable many bounded and continuous functions $a_{i,t}(x),b_{i,t}(x)$ such that
\[E(w)(x,t)=\sup_{i\in \mathbb{N}}\left\{a_{i,t}(x)+b_{i,t}(x)w\right\}.
\]
Let $w(x,t)=(-\nabla V(x)-\nabla W*\mu) (x,t)$ and similarly for $w^n(x,t)$. Then for each ${N\in\mathbb{N}}$ and $t$, \begin{align*}
\liminf_n \int_{ \mathbb{R}^d \cap \{c\geq 0\}}E(w^n)d\mu^n &\geq
    \liminf_n \int_{ \mathbb{R}^d \cap \{c\geq 0\}}\max\{(a_{i,t}+b_{i,t}w^n), 0\leq i\leq N\}d\mu^n\\
    &=\int_{ \mathbb{R}^d \cap \{c\geq 0\}}\max\{(a_{i,t}+b_{i,t}w),0\leq i\leq N\}d\mu. 
\end{align*}
Then take sup over $i\in\mathbb{N}$ and use Lebesgue's monotone convergence theorem. We see
\begin{equation}\label{cl}
    \liminf_n \int_s^t\int_{ \mathbb{R}^d \cap \{c\geq 0\}}E(w^n)d\mu^n dr \geq \int_s^t\int_{ \mathbb{R}^d \cap \{c\geq 0\}}E(w)d\mu dr.
\end{equation}

If $W$ is $C^2$, we can show the lower semi-continuity of $\mathcal{E}$ without non-negative assumptions on $c$. We want to show that $E(w)(x,t)$ is lower semi-continuous on $w,x,t$. Continuity on $w$ is clear by definition and we have uniform continuity for all $(x,t)\in supp(\mu(t))\subseteq {\Omega_T}$ since the support is compact and $c$ is continuous. Let $\gamma(y,y')$ be the optimal transport plan between $\mu^n$ and $\mu$. By \textsc{(C2)}, $|D^2W|$ is bounded in compact sets. By definition
\begin{align*}
    |w^n(x,t)-w(x,t)|&=\left|\int_{\mathbb{R}^d}\nabla W(x-y) d\l\mu^n(y,t)-\mu(y,t)\r\right|\\
    &\leq \int_{\mathbb{R}^{2d}}|\nabla W(x-y)-\nabla W(x-y')|d\gamma(y,y')\\
    &\leq C \|D^2W\|_{L^\infty( supp\{\mu\})} \int_{\mathbb{R}^{2d}}|y-y'|d\gamma(y,y')\\
    &\leq C d_W(\mu^n,\mu).
\end{align*}
So for all $\epsilon>0$, if $n$ is large enough, we have for $t\in[0,T]$ 
\begin{equation}\label{clq0}
\int_{\mathbb{R}^{d}}E(w^n)d\mu^n\geq \int_{\mathbb{R}^{d}}\l E(w)-\epsilon \r d\mu^n\geq \int_{\mathbb{R}^{d}} E(w) d\mu^n-\epsilon .
\end{equation}

Now we prove lower semi-continuity of $E(\cdot)$ in $(x,t)$. For $(x,t)$ inside the domain
$E(w)(x,t)=\frac{1}{2}w^2\geq 0.
$
But for $(x,t)$ on the boundary
\[
 E(w)(x,t)=\begin{cases}\frac{1}{2}\left|w\right|^2 &\text{ if }w^N=w\cdot n\leq c,\\
c w^N-\frac{1}{2}c^2 &\text{ if }w^N> c.
\end{cases}
\]
And there is $c w^N-\frac{1}{2}c^2\leq \frac{1}{2}\l w^N \r^2\leq \frac{1}{2}\left|w\right|^2$. From this, $E(w)(x,t)$ is lower semi-continuous in $(x,t)$. Combining that $\mu^n$ converges narrowly to $\mu$, we have
\[ \liminf_{n\rightarrow \infty}\int_{\mathbb{R}^{d}} E(w) d\mu^n \geq \liminf_{n\rightarrow \infty}\int_{\mathbb{R}^{d}} E(w) d\mu.
\]
Take $\epsilon\to 0$ in \eqref{clq0} and integrate in time, we proved
\[ \liminf_{n\rightarrow \infty}\int_s^t\int_{\mathbb{R}^{d}} E(w^n) d\mu^ndr \geq \liminf_{n\rightarrow \infty}\int_s^t\int_{\mathbb{R}^{d}} E(w) d\mu dr.
\]

\end{proof}

Now we give the main theorem:
\begin{theorem}\label{thmu}
Assume conditions \textsc{(O1)(C1)-(C3)} hold. Suppose $\mu_0$ is a probability measure compactly supported in $\overline{\Omega(0)}$ and $\{ \mu^n(\cdot)\}$ are the discrete solutions as stated above. 
Then $\mu^n(\cdot)$ converges to $\mu(\cdot)$ which is an absolutely continuous curve in $\mathcal{P}_2(\mathbb{R}^d)$ of maximal slope with respect to $\phi$ (see definition \ref{extra}). And $\mu(\cdot)$ satisfies the equation \eqref{eqn1} in the weak sense.

Suppose we have two such solutions $\mu^i(\cdot), i=1,2$ with initial data $\mu^i_0\in \mathcal{P}_2(\mathbb{R}^d)$, and $\mu^i(\cdot)$ are compactly supported for $t\in[0,T]$. Then there exists a constant $C$ that 
\[d_W\left(\mu^1(t),\mu^2(t)\right)\leq C d_W(\mu^1_0,\mu^2_0) \text{ for all }0\leq t\leq T.\]
Here $C$ depends on universal constants and the support of $\mu^i_0$.
\end{theorem}

\begin{proof}
First we show $t\rightarrow \phi\left(\mu^n(t)\right)$ is absolutely continuous:
\begin{align*}
    &\left|\phi\left(\mu^n(t)\right)-\phi\left(\mu^n(s)\right) \right|\leq \sum_j m_j \left|V\left(x_j(t)\right)-V\left(x_j(s)\right)\right|+\\
    &\quad \sum_{i<j}m_i m_j \left|W\left(x_i(t)-x_j(t)\right)-W\left(x_i(s)-x_j(s)\right)\right|\\
\leq &\sum_j m_j|\nabla V(\theta)||x_j(t)-x_j(s)|+\sum_{i<j}m_i m_j |\nabla W(\theta')|(|x_j(t)-x_j(s)|+|x_i(t)-x_i(s)|)
\end{align*}
where $\theta=h x_j(t)+(1-h)x_j(s)$ and $\theta'=h' \left(x_i(t)-x_j(t)\right)+(1-h')\left(x_i(s)-x_j(s)\right)$ for some $0\leq h,h'\leq 1$. By \textsc{(C1)(C2)}, $|\nabla V(\theta)|,|\nabla W(\theta')|$ is bounded by $C(1+|x_j(t)|+|x_j(s)|)$, $C(1+|x_j(t)|+|x_j(s)|+|x_i(t)|+|x_i(s)|)$ respectively. So 
\[
\left|\phi\left(\mu^n(t)\right)-\phi\left(\mu^n(s)\right) \right|\leq C\left(1+m_2\left(\mu^n(t)\right)+m_2\left(\mu^n(s)\right)\right)d_W\left(\mu^n(t),\mu^n(s)\right).\]
Thus $\phi\left(\mu^n(t)\right)$ is absolutely continuous in $t$ since $\mu^n(\cdot)$ is absolutely continuous curves and $\{m_2(\mu^n)\}$ are uniformly bounded. We do not need $\mu^n$ to be compactly supported here.

Then we show that the discrete solutions $\mu^n$ are curves of maximal slope. By direct computation
\begin{align*}
    \frac{d}{d t}\phi\left(\mu^n(t)\right)&=\sum_j m_j\nabla V(x_j(t))\cdot \dot{x_j}(t)+\frac{1}{2}\sum_i\sum_j m_i m_j\nabla W(x_i(t)-x_j(t))\cdot (\dot{x_i}(t)-\dot{x_j}(t))\\
    &=\sum_i m_i\nabla V(x_i(t))\cdot \dot{x_i}(t)+\sum_i\sum_j m_i m_j\nabla W(x_i(t)-x_j(t))\cdot (\dot{x_i}(t)) \text{  (by symmetry of $W$)}\\
&=-\int_{ \mathbb{R}^d }|v^n|^2d\mu^n(t)+\int_{ \mathbb{R}^d }(v^n-w^n)\cdot v^n d\mu^n(t).
\end{align*}
In the above we used the notation \eqref{vpf}.
By Proposition \ref{thm1} \[|(\mu^n)'|^2(t)\leq \int_{\Omega(t)}|v^n|^2 d\mu^n \text{ a.e. in time},\] 
we deduce that $\mu^n(\cdot)$ is a curve of maximal slope according to the definition \eqref{extra}:
\begin{equation}
\phi\left(\mu^n(s)\right)\geq \phi\left(\mu^n(t)\right)+\frac{1}{2}\int^t_s|(\mu^n)'|^2(r)dr+\mathcal{E}(\mu^n).\label{dn}
\end{equation}

Now we want to show that $\mu(t)$ is also a curve of maximal slope. By Theorem 3.6 \cite{prox},
\begin{equation}\label{u'9}
\liminf_n\int^t_s|(\mu^n)'|^2(r)dr\geq \int^t_s|\mu'|^2(r)dr.
\end{equation}
Because $\mu^n(t),\mu(t)$ are compactly supported locally uniformly in time, $W,V\in C^1(\mathbb{R}^d)$ and $\mu^n(t)\to\mu(t)$ in $d_W$, we have for each $t$,
\begin{equation}
    \label{phi n convergence}
    \lim_{n\to\infty}\phi(\mu^n(t))=\phi(\mu(t)).
\end{equation}

Recall the notation \begin{equation}\label{knpf}
k(x,t) n(x,t)=w(x,t)- P_{x,t} w(x,t).
\end{equation}
Then by \eqref{dn} \eqref{u'9} and Lemma \ref{lemu}, sending $n$ to infinity shows
\[
\phi\left(\mu(s)\right)\geq \phi\left(\mu(t)\right)+\frac{1}{2}\int^t_s|\mu'|^2(r)dr+\frac{1}{2}\int^t_s\int_{ \mathbb{R}^d }|v|^2(x,r)d\mu dt\]\[+\int^t_s\int_{ \mathbb{R}^d }k(x,t)n(x,r)\cdot v(x,r) d\mu dt.
\]
Notice $k\geq 0$ and if $k(x,t)>0$ \[n(x,t)\cdot v(x,t)=n(x,t) \cdot  P_{x,t} w(x,t)=c(x,t).\] This gives
\begin{equation}\label{phig}
\begin{aligned}
&\phi\left(\mu(s)\right)\geq \phi\left(\mu(t)\right)+\frac{1}{2}\int^t_s|\mu'|^2(r)dr+\frac{1}{2}\int^t_s\int_{ \mathbb{R}^d }|v|^2d\mu dr\\
&\quad +\int_s^t\int_{ \mathbb{R}^d }k(x,r)c(x,r)d\mu dr.
\end{aligned}
\end{equation}

Next because $\mu(\cdot)$ is an absolutely continuous curve in $\mathcal{P}_2( \mathbb{R}^d) $, by Theorem 8.3.1 \cite{gflow}, there exists a unique tangent vector field $\hat{v}(x,t)\in L^2(\mu(t),\mathbb{R}^d)$ such that the continuity equation holds
\[ \frac{\partial}{\partial t}\mu +\nabla\cdot(\hat{v}\mu)=0\]
\begin{equation}\label{vh}
\text{and  }
\int_{ \mathbb{R}^d }|\hat{v}|^2d\mu(x,t)=|\mu'|^2(t).\end{equation}
The goal is to show $v=\hat{v}$. 

\medskip
We \textbf{claim} the following chain rule that for a.e. $t>0$
\[
\frac{d}{d t}\phi\left(\mu(t)\right)=-\int_{ \mathbb{R}^d }w(x,t)\cdot \hat{v}(x,t)d\mu(x,t).
\]
A similar result is proved in Theorem 1.5 \cite{prox}. For the convenience of readers, we give a sketch of proof below. For $\tau>0$, select $\gamma_t^\tau\in \Gamma_0\left(\mu(t),\mu(t+\tau)\right)$. Then
\begin{align*}
    &\phi(\mu(t+\tau))-\phi(\mu(t))\\
    =&\int_{\mathbb{R}^{2d}}V(y)-V(x)d \gamma^\tau_t(x,y)+\frac{1}{2}\int_{\mathbb{R}^{4d}}W(y_2-y_1)-W(x_2-x_1)d \gamma^\tau_t(x_1,y_1)d \gamma^\tau_t(x_2,y_2)\\
    \geq &\int_{\mathbb{R}^{2d}}\langle\nabla V(x),y-x\rangle -C|y-x|^2d \gamma^\tau_t(x,y)+\\
    &\quad \quad\int_{\mathbb{R}^{4d}}\left(\langle(\nabla W)(x_2-x_1), y_2-x_2\rangle d \gamma^\tau_t(x_1,y_1)-C|y_2-x_2|^2\right)d \gamma^\tau_t(x_2,y_2)\\
    =&\int_{\mathbb{R}^{2d}}\langle\nabla V(x),y-x\rangle d \gamma^\tau_t(x,y)+\int_{\mathbb{R}^{4d}}\langle(\nabla W *\mu(\cdot, t))(x_2), y_2-x_2\rangle d \gamma^\tau_t(x_2,y_2)-C' d_W^2(\mu(t),\mu(t+\tau))
\end{align*}
We used (C2)(C3) ($\lambda$-convexity of $V,W$ and symmetry of $W$) in the above inequality and the constants $C,C'$ only depend on $\lambda$. Then applying that $\mu(t)$ is absolutely continuous, for a.e. $t>0$
\begin{align*}
    &\lim_{\tau\to 0^+}\frac{\phi(\mu(t+\tau)-\phi(\mu(t)))}{\tau}\\
    \geq & \limsup_{\tau\to 0^+}\left( \int_{\mathbb{R}^{2d}}\langle \nabla V(x)+\nabla W *\mu(x),\frac{y-x}{\tau}\rangle d \gamma^\tau_t(x,y)-\frac{C'}{\tau}d_W^2(\mu(t),\mu(t+\tau))\right).
\end{align*}
By Proposition 8.4.6 \cite{gflow}, $\int_{\Omega(t)}\frac{y-\cdot}{\tau}dr^\tau(y)$ converges to $\hat{v}(\cdot,t)$ weakly in $L^2(\mu(t))$ where $r^\tau(y)$ is the disintegration of $\gamma_t^\tau$ with respect to $\mu(t)$. Note $\mu(t)$ is absolutely continuous, therefore for a.e. $t>0$
\begin{align*}
    \lim_{\tau\to 0^+}\frac{\phi(\mu(t+\tau)-\phi(\mu(t)))}{\tau}
    \geq  \int_{\mathbb{R}^{2d}}\langle \nabla V(x)+\nabla W *\mu(x),\hat{v}(x,t)\rangle d \mu(t).
\end{align*}
Similarly, we have
\begin{align*}
    \lim_{\tau\to 0^+}\frac{\phi(\mu(t)-\phi(\mu(t-\tau)))}{\tau}
    \leq  \int_{\mathbb{R}^{2d}}\langle \nabla V(x)+\nabla W *\mu(x),\hat{v}(x,t)\rangle d \mu(t).
\end{align*}
Again since $\mu(t)$ is absolutely continuous, as done in the discrete case (see Theorem \ref{thmu}) we have $t\to \phi(\mu(t))$ is absolutely continuous. Thus, we can conclude with the chain rule. 

\medskip
Then using the notation \eqref{knpf}
\begin{equation}\label{phil}\phi\left(\mu(s)\right)-\phi\left(\mu(t)\right)=\int^t_s\int_{ \mathbb{R}^d }v(x,r)\cdot \hat{v}(x,r)d\mu(r)dr+\int^t_s\int_{ \mathbb{R}^d }k(x,r)n(x,r)\cdot \hat{v}(x,r)d\mu(r)dr
\end{equation}
Recall $r^\tau$ as defined above. We have
\[\int^t_s\int_{ \mathbb{R}^d }k(x,r)n(x,r)\cdot \hat{v}(x,r)d\mu(r)dr=\]
\[\int^t_s\int_{\Omega(r)}k(x,r)\lim_{\tau\rightarrow  0} \int_{\Omega(r+\tau)}n(x,r)\cdot \frac{y-x}{\tau}dr^\tau(y)du(x,r).\]
Note in the above equation, $y\in\Omega(r+\tau),x\in\partial\Omega(r)$. Take $x'(x,r,\tau)$ to be one of the closest point to $x$ on $\partial\Omega(r+\tau)$. So for $\tau$ small enough, $\frac{(x'-x)}{\tau} = c(x,r)n+O(\tau)$. Also since $c(x,r)$ is continuous and $\mu(r)$ is compactly supported, we get
\[\int^t_s\int_{ \mathbb{R}^d }k(x,r)n(x,r)\cdot \hat{v}(x,r)d\mu(r)dr\leq \int^t_s\int_{ \mathbb{R}^d }k(x,r)c(x,r)d\mu(r)dr\]
\begin{equation}+\limsup_{\tau\rightarrow  0}\int^t_s\int_{\Omega(r)} \int_{\Omega(r+\tau)}k(x,r)n(x,r)\cdot \frac{y-x'}{\tau}dr^\tau(y)d\mu(r)dr.\label{tauk}
\end{equation}
Now by $r_p$-prox regularity of $\Omega(r)$, within the support of $\mu$,
\begin{align*}
    k(x,r)n(x,r)\cdot \frac{y-x'}{\tau}&\leq \frac{C_k}{r_p\tau}|y-x'|^2\\
    &\leq \frac{C_k}{r_p\tau}|y-x|^2+\frac{C_k}{r_p\tau}|x-x'||2y-x-x'|\\
&\leq \frac{C_k}{r_p\tau}|y-x|^2+\frac{2C_k}{r_p}|c(x)||y-x|+o(\tau).
\end{align*}
$C_k$ here is the bound of $k(x,r)$ for $(x,r)\in\overline{\Omega}_T$ and it depends on the support of $\mu_0$.
By the dominated convergence theorem the last term of (\ref{tauk})
\begin{align*}
    &\leq \limsup_{\tau\rightarrow  0}\int^t_s\int_{\Omega(r)\times {\Omega(r+\tau)}} \frac{C}{\tau}|y-x|^2d\gamma_r^\tau dr+\limsup_{\tau\rightarrow  0}\int^t_s\left(\int_{\Omega(r)\times {\Omega(r+\tau)}} C|y-x|^2d\gamma_r^\tau\right)^\frac{1}{2}dr\\
&\leq \limsup_{\tau\rightarrow  0}\int^t_s\frac{C}{\tau}d_W^2\left(\mu(r),\mu(r+\tau)\right)dr+\limsup_{\tau\rightarrow  0}\left(\int^t_s\int_{\Omega(r)\times {\Omega(r+\tau)}} C|y-x|^2d\gamma_t^\tau dr\right)^\frac{1}{2}\\
&\leq \limsup_{\tau\rightarrow  0}\frac{C}{\tau}\int^t_s d_W^2\left(\mu(r),\mu(r+\tau)\right)dr+\limsup_{\tau\rightarrow  0}C\left(\int^t_s d_W^2(\mu(r),\mu(r+\tau)dr\right)^\frac{1}{2}.
\end{align*}
In view of the fact that $\mu(\cdot)$ is an absolutely continuous curve with respect to the Wasserstein metirc, we have the above expressions vanish. So
\begin{equation}\label{vhat}
\int^t_s\int_{ \mathbb{R}^d }k(x,r)n(x,r)\cdot \hat{v}(x,r)d\mu(r)dr\leq \int^t_s\int_{ \mathbb{R}^d }k(x,r)c(x,r)d\mu(r)dr.
\end{equation}
Apply this to \eqref{phil}, we get
\begin{equation}\label{phil1}
\phi\left(\mu(s)\right)\leq \phi\left(\mu(t)\right)+\int^t_s\int_{ \mathbb{R}^d }v(x,r)\cdot \hat{v}(x,r)d\mu(r)dr+\int^t_s\int_{ \mathbb{R}^d }k(x,r)c(x,r)d\mu(r)dr.
\end{equation}
Finally compare \eqref{phig} with \eqref{phil1} and make use of \eqref{vh}, we find in $\overline{\Omega}_T$ a.e. $d\mu dt$
\begin{equation}\label{hateq}
\hat{v}(x,t)=v(x,t)= P_{x,t} w(\mu)(x,t).
\end{equation}

For the stability result, considering that $\mu_1, \mu_2$ are compactly supported, the proof is essentially the same as the one in Proposition \ref{thm2}.
\end{proof}

\begin{remark}
As can be seen in the proof, for Theorem \ref{thmu} we only need $\lambda$-convexity of $V,W$ locally. Also we can weaken the condition \textsc{(O1)} to be local prox-regularity: for every ball $B_R$, $\partial \Omega(t)\cap B_R$ is $r_p$-prox-regular for some $r_p(R)>0$.
\end{remark}

\begin{remark}\label{rmkunique}
From the theorem we have the uniqueness of solutions to \eqref{eqn1} with compact support for any finite time. But it is not clear to us whether solutions can spread out to infinity far of the domain within a finite time, even with compactly supported initial data. It is unknown to us about the general uniqueness result.
\end{remark}

\begin{remark}
If the domain is non-compressing which is equivalent to $c\geq 0$, we don't require $W\in C^2$. Lemma \ref{lemu} is the only place in part one of the paper where we need the assumption.
\end{remark}

\subsection{Non-Compactly Supported Data}\label{sub2.4}

In this section we consider the equation \eqref{eqn1} with non-compactly supported initial data. 
Let $\mu\in\mathcal{P}^2(\mathbb{R}^d)$, define
\[\delta(R,\mu):= \left\{\int_{B_R^C}|x|^2 d\mu\right\}\] where $B_R^C$ denotes the complement of $B_R$. We say $\mu$ satisfies condition \textsc{(R)} if

\medskip

\begin{flushleft}
\textbf{(R)} There exists some constant $c_r>0$ such that $\delta(R,\mu) exp(c_rR)\rightarrow 0$ as $R\rightarrow +\infty$.
\end{flushleft}
This condition requires some exponentially decay of measures which is slightly more general than compact supported ones. We say a curve of measures $\{\mu(t),t\geq 0\}$ satisfy condition \textsc{(R)} locally uniformly if for each $t\in [0,T]$, $\mu(t)$ satisfies the condition for some $c_r(T).$

\begin{theorem}\label{converge}
Suppose conditions \textsc{(O1)(C1)-(C3)} hold and $\mu_0$ satisfies \textsc{(R)}. Then there exists a weak solution $\mu(\cdot)$ to equation \eqref{eqn1}. And it is an absolutely continuous curve in $\mathcal{P}_2(\mathbb{R}^d)$ which satisfies condition \textsc{(R)} locally uniformly.

Suppose $\mu^i(\cdot), i=1,2$ are two solutions with initial data $\mu^i_0$ satisfying \textsc{(R)} locally uniformly with constant $c_r$. Then for any $0<p<1$, there is $t_p,\varepsilon>0$ that for all $0\leq t\leq t_p$, $d_W(\mu^1_0,\mu^2_0)\leq \varepsilon$ we have
\[d_W\left(\mu^1(t),\mu^2(t)\right)\leq 2d_W(\mu^1_0,\mu^2_0)^p.\]
Here $t_p,\varepsilon$ only depend on $p,c_r$ and universal constants.
\end{theorem}

\begin{proof}
For existence, we use the particle approximation method as before: let $\mu^n(\cdot)=\sum_{i=1}^{k(n)}m_i \delta_{x_i^n(\cdot)}$ be solutions to equation \eqref{eqn1} with discrete initial data $\mu^n_0\rightarrow \mu_0$. First let us assume the convergence of $\mu^n(\cdot)$ and show the limit is a solution. Expressions or estimates \eqref{cl} \eqref{dn} \eqref{u'9} still hold. But we need to be careful on \eqref{phi n convergence} \eqref{vhat} and Lemma \ref{lemu} since the solutions are no longer supported in a compact set. If \eqref{phi n convergence}, \eqref{lemu1} and \eqref{vhat} are valid, we deduce \eqref{phig} \eqref{hateq} and from which we draw the conclusion.

To show \eqref{lemu1}, we use truncation method. Recall estimate \eqref{bound x2}, within time $T$, we get
\begin{equation}
    \label{delta R uniform}
    \delta(R,\mu^n(t))\leq \delta(\frac{R}{C_T},\mu^n_0)
\end{equation} which converges to $0$ exponentially fast as $R\rightarrow \infty$. 

By \eqref{bound F}\eqref{bound x2} and definition of $ P_{x,t} $
\[|k^n(x,t)|\leq C(1+|x|),\quad | P_{x,t} w^n(x,t)|\leq |w^n(x,t)|+|c(x,t)|\leq C(1+|x|).\]
And similar linear bounds also hold for $k(x,t), P_{x,t} w (x,t)$. So for any small $\epsilon>0$, we can choose $R$ large enough such that
\begin{equation*}\label{BR}
\left|\int_s^t\int_{ B_R^C}E(w^n)d\mu^ndr\right|\leq C\int_{B_R^C}|x|^2d\mu^ndr<\epsilon
\end{equation*}
for all $n$. Then we only need to consider $
\tilde{\mu}^n(t)=\sum_j m_j\delta_{\{x_j(t),|x_j(0)|\leq R\}}$. The contribution for particles starting outside $B_R$ will be under control by \eqref{bound x2} again. Now since the integration is inside a compact set, the proof will then follows from Lemma \ref{lemu}. The proof of \eqref{vhat} is similar. 

For \eqref{phi n convergence}, let us only write down the proof of $\lim_{n\to\infty}\mathcal{W}({\mu^n(t)})=\mathcal{W}({\mu(t)})$. Recall that $\gamma$ denotes the optimal transport plan between $\mu^n(t)$ and $\mu(t)$. For any $\epsilon>0$
\begin{align*}
    \left|\mathcal{W}({\mu^n(t)})-\mathcal{W}({\mu(t)})\right|&\leq \frac{1}{2} \int_{\mathbb{R}^{4d}}|W(x-x')-W(y-y')|d\gamma(x,y)d\gamma(x',y')\\
    &\leq C\int_{\mathbb{R}^{4d}}(1+|x|+|x'|+|y|+|y'|)(|x-x'-y+y'|)d\gamma(x,y)d\gamma(x',y')\\
    &\leq \epsilon \int_{\mathbb{R}^{4d}}(1+|x|^2+|x'|^2+|y|^2+|y'|^2)d\gamma(x,y)d\gamma(x',y')\\
    &\quad +C_\epsilon \int_{\mathbb{R}^{4d}}(|x-y|^2+|x'-y'|^2)d\gamma(x,y)d\gamma(x',y') \\
    &\leq \epsilon(1+2m_2(\mu^n(t))+2m_2(\mu(t)))+2C_\epsilon d^2_W(\mu^n(t),\mu(t)).
\end{align*}
In the above we used condition (C2). By \eqref{delta R uniform}, we know that the second moment of $\mu^n,\mu$ are bounded locally uniformly in time. Then if $\mu^n(t)$ converges $\mu(t)$ in $d_W$, we deduce
$\lim_{n\to\infty}\mathcal{W}({\mu^n(t)})=\mathcal{W}({\mu(t)})$. Similarly $\lim_{n\to\infty}\mathcal{V}({\mu^n(t)})=\mathcal{V}({\mu(t)})$ and \eqref{phi n convergence} follows.

\medskip

Now we show the convergence of $\mu^n(\cdot)$. We use the notation $\gamma^B_A$ as the restriction of $\gamma\in \mathcal{P}(\mathbb{R}^{2d})$ in $A\times B \subset  \mathbb{R}^d \times  \mathbb{R}^d $. Also we denote $A^C$ as the complement of $A$ in $ \mathbb{R}^d $ while $(\gamma^B_A)^C$ as the restriction of $\gamma$ on $(A\times B)^C$.
Without loss of generality, assume that $\delta(R,\mu^n_0)$ are comparable to $\delta(R,\mu_0)$ for all $n$ and $R>1$. 
For simplicity of notation, we write $\gamma_t$ as the optimal transport plan between $\mu^n(t)$ and $\mu^m(t)$. Similarly as before, we have
\begin{align}
\frac{1}{2}\frac{d}{d t}d_W^2\left(\mu^n(t),\mu^m(t)\right)&=\int_{\mathbb{R}^{2d}}\langle v^n(x,t)-v^m(y,t),x-y\rangle d\gamma_t(x,y)\nonumber\\
&\leq Cd_W^2(\mu^n,\mu^m)-\int_{\mathbb{R}^{2d}}\langle k^n(x,t) n(x,t)-k^m(y,t)n(y,t),x-y\rangle d\gamma_t(x,y).\label{kterm}
\end{align}
By prox-regularity, the last of \eqref{kterm}
\[
\leq \int_{\mathbb{R}^{2d}}\min \left\{\frac{|k^n|+|k^m|}{r_p}|x-y|^2,(|k^n|+|k^m|)|x-y|\right\}d\gamma_t(x,y).\]
For any $L>>1$, let $C'$ be the constant as given in \eqref{bound x2} and $L':=C'L$. 
Then the above
\begin{equation*}\label{2part}
\leq \int_{\mathbb{R}^{2d}} \frac{|k^n|+|k^m|}{r_p}|x-y|^2d\gamma^{B_{L'}}_{B_{L'}}
+\int_{\mathbb{R}^{2d}}(|k^n|+|k^m|)|x-y|d(\gamma^{B_{L'}}_{B_{L'}})^C.\end{equation*}
By \eqref{bound k}, 
\[\int_{\mathbb{R}^{2d}} \frac{|k^n|+|k^m|}{r_p}|x-y|^2d\gamma^{B_{L'}}_{B_{L'}}
\leq CLd_W^2(\mu^n,\mu^m),\]
\[\int_{\mathbb{R}^{2d}}(|k^n|+|k^m|)|x-y|d(\gamma^{B_{L'}}_{B_{L'}})^C
\leq C\int_{\mathbb{R}^{2d}}(|x|^2+|y|^2)d(\gamma^{B_{L'}}_{B_{L'}})^C
\]
\begin{equation}\label{4term}
\leq C\left(\int_{|x|\geq {L'}} |x|^2d\mu^n+\int_{|y|\geq {L'}} |y|^2d\mu^m+m_2(\mu^m)\int_{|x|\geq {L'}} 1d\mu^n +m_2(\mu^n)\int_{|y|\geq {L'}} 1d\mu^m \right)
\end{equation}
By \eqref{bound x2}, for all $0\leq t\leq T$, if $x_i(t)\in B_{L'}^C$, then $x_i(0)\in B_{L}^C$ and $|x_i(t)|\leq C'(|x_{i}(0)|+1)$. This, combining with \eqref{boundm2}, gives 
\[\eqref{4term}\leq CC'^2\left(\int_{|x|\geq L} |x|^2d\mu^n(0)+\int_{|y|\geq L} |y|^2d\mu^m(0)\right)\leq CC'^2\delta(L,\mu_0).\]
Note $C'$ is a universal constant. In all, finally we have
\[
\frac{d}{d t}d_W^2(\mu^n,\mu^m) \leq CLd_W^2(\mu^n,\mu^m) + C\delta(L,\mu_0).
\]
This shows
\begin{equation*}\label{cauchy}
d_W^2\left(\mu^n(t),\mu^m(t)\right)\leq C\delta(L,\mu_0) \frac{\exp(CLt)-1}{L}+d_W^2(\mu^n_0,\mu^m_0)\text{ } \exp(CLt).
\end{equation*} Select $t_0=\frac{c_r}{C}$ where $c_r$ comes from \textsc{(R)}. 
By the condition, $\delta(L,\mu_0)\exp(c_rL)$ can be any small if $L$ is large. Recall that $\{\mu^n_0\}$ is Cauchy and we further take $d_W(\mu_0^n,\mu_0^m)$ to be small, thus $\{\mu^n(t)\}$ is also Cauchy for $0\leq t\leq t_0$. Then we can consider each time interval: $[0,t_0],[t_0,2t_0]...$ inductively and we proved that $\mu^n(t)\rightarrow \mu(t)$ for all $t\leq T$. 

\medskip

For any $0< p<1$, write $\varepsilon:=d_W(\mu^n_0,\mu^m_0)$ and choose $L=\frac{-2 \log\varepsilon}{c_r}$, $t_0$ as above. For $t\leq t_0(1-p)$, 
\[d_W^2\left(\mu^n(t),\mu^m(t)\right)\leq C\delta(L,\mu_0) \exp(c_rL-c_rpL)+\varepsilon^{2p}.\]
Let $\varepsilon$ be small enough and $L$ is then large enough. By \textsc{(R)},
\[d_W^2\left(\mu^n(t),\mu^m(t)\right) \leq C\delta(L,\mu_0) \exp(c_rL)\varepsilon^{2p}+\varepsilon^{2p} \leq 4d_W^{2p}(\mu^n_0,\mu^m_0).\]

\medskip

Notice if $\mu(\cdot)$ solves the equation \eqref{eqn1},
then $m_2(\mu(t))\leq C$ for $t\in[0,T]$. The second claim about the stability of solutions satisfying condition \textsc{(R)} follows from the above argument for the discrete type solutions. 
This shows that solutions satisfying condition \textsc{(R)} are unique.
\end{proof}

\begin{remark}\label{convex}
We comment on several situations where the condition (R) can (or possibly) be dropped.
\begin{flushleft}
(i) As in Theorem 1.9 \cite{prox}, if $\Omega(t)$ is convex for all $t$ and $\mu^i(t)$($i=1,2$) are solutions with general initial data $\mu^i_0\in\mathcal{P}_2(\overline{\Omega(t)})$, then there exists a universal constant $C$ such that
\[d_W\left(\mu^1(t),\mu^2(t)\right)\leq Cd_W(\mu^1_0,\mu^2_0).\]
Here we do not need any assumptions on the decay of solutions. 
The proof follows from the observation that in \eqref{kterm}
$\langle k(x,t)n(x,t),x-y\rangle\geq 0 $ by the convexity. And as a corollary we have the uniqueness result.\\

(ii) If $\nabla V,\nabla W$ and $c(x,t)$ are uniformly bounded, we can conclude the same as in (i).\\

(iii) We guess that for the non-local term if $ W$ is compactly supported, there is a better stability result. 

\end{flushleft}
\end{remark}

\subsection{Examples and Stability of Solutions}

In this section, we will show that the stability result in Theorem \ref{converge} cannot be improved to 
$
d_W\left(\mu^1(t),\mu^2(t)\right)\leq Cd_W(\mu^1_0,\mu^2_0)$
as long as the domain is unbounded and non-convex.
Moreover we give examples showing that without condition \textsc{(R)}, the stability of solutions is even weaker than that in Theorem \ref{converge}. All these suggest that the stability of solutions to \eqref{eqn1} is strongly related to the decay of initial data at infinity. 

\begin{theorem}\label{example}
There exists $V,W,\Omega$ satisfy conditions \textsc{(C1)-(C3)(O1)} such that the following holds for any $t_0>0$. 
\begin{flushleft}
(i) There is $\mu_0\in\mathcal{P}_2(\overline{\Omega})$ and $\mu^n_0\rightarrow \mu_0$ (write $\mu^n(t)$ as a solution to equation \eqref{eqn1} with initial data $\mu^n_0$) that
$\mu(\cdot),\mu^n(\cdot)$ satisfy \textsc{(R)} locally uniformly and
\[
\liminf_n d_W\left(\mu(t_0),\mu^n(t_0)\right)/d_W(\mu_0,\mu^n_0)=+\infty.\]

(ii) For any $p>\frac{1}{2}$, there is $\mu_0\in\mathcal{P}_2(\overline{\Omega})$ and $\mu^n_0\rightarrow \mu_0$ that
\[
\liminf_n d_W\left(\mu(t_0),\mu^n(t_0)\right)/d_W^p(\mu_0,\mu^n_0)=+\infty.\]
\end{flushleft}
\end{theorem}

\begin{proof}
Consider the following stationary domain in $\mathbb{R}^2$
\[\Omega:=\{(x,y)| \quad y\geq \cos (2\pi x)\}\]
and the equation
\begin{equation}\label{eqexam}
\frac{\partial}{\partial t}\mu+\nabla\cdot (\mu P(-\nabla V) )=0\quad \text{with } V=xy.
\end{equation}
Here $P$ is the projection operator defined in \eqref{Pt} with zero boundary speed.

Let us start by solving the equation with initial data $\mu_0=\delta_{(x_0,\cos(2\pi x_0))}$ where $x_0$ is close to some large integer. Suppose $\delta_{(x_t,y_t)}$ is a solution, then by simple calculations 
\[(x'_t,y'_t)=P(-\nabla V(x_t,y_t)=P(-y_t,-x_t).\]
It is not hard to see that if $x_t$ is large enough, the delta mass moves along the boundary. So
$P(-y_t,-x_t)=P(- \cos(2\pi x_t),-x_t)$ and
\begin{equation}
v_t:=(x'_t,y'_t)=\left(\vartheta_t,-2\pi \sin(2\pi x_t)\vartheta_t\right) \text{ where }\vartheta_t=\frac{-(\cos(2\pi x_t)-2\pi x_t \sin(2\pi x_t))}{1+4\pi^2 \sin^2(2\pi x_t)}.\label{Nsim}
\end{equation} 
Notice the equation $x_t'=\vartheta_t$ is an autonomous system. For each large integer $N$, there are two equilibrium points in $[N-\frac{1}{2},N+\frac{1}{2}]$ namely the solutions of $2\pi x=\cot(2\pi x)$. One equilibrium point is close to $N-\frac{1}{2}$ which is stable and the other is close to $N$ which is unstable. We denote the stable one as $ (N)_*^1$ and the other as $(N)_*^2$. We can show $(N)_*^1- (N-\frac{1}{2})\approx  (N)_*^2-N\approx \frac{1}{N}$.

Now consider
\[\mu_0=\sum_{j=1}^{\infty} m_j \delta_{\left(j,\cos(2\pi j)\right)}.\]
Also we construct a family of $\mu^n_0\rightarrow\mu_0$. Denote $j_\alpha:= j+(j)^{-\alpha}$ with $0<\alpha<1$ and let
\[\mu^n_0=\sum_{j=1}^{n-1} m_j \delta_{(j,\cos(2\pi j))}+\sum_{j=n}^{\infty} m_j \delta_{\left(j_\alpha,\cos(2\pi j_\alpha)\right)}.\]

As usual, we have the solutions $\mu^n(t)=\sum_j\delta_{(x_j^n(t),y_j^n(t))}$ and $\mu(t)=\sum_j\delta_{(x_j(t),y_j(t))}$ with $\mu^n(t)=\mu_0^n,\mu(t)=\mu_0$. So for any fix $t_0>0$:
\[
d_W^2\left(\mu(t_0),\mu^n(t_0)\right)/d_W^2(\mu_0,\mu^n_0)\gtrsim \]
\[\sum_{j\geq n} m_j\left(\left|x_j(t_0)-x_j^n(t_0)\right|^2+\left|y_j(t_0)-y_j^n(t_0)\right|^2\right)/\left(\sum_{j\geq n} m_j (3j)^{-2\alpha}\right).\]
Because $\alpha<1$ and $j_\alpha>(j)_*^1$, for $j$ large enough we have $(x_j^n(t),y_j^n(t))$ starts at $\left(j_\alpha,\cos(2\pi j_\alpha)\right)$ and moves towards $((j+1)_*^1,\cos(2\pi (j+1)_*^1))$. However the mass starting from $(j,\cos(2\pi j))$ will move towards $((j)^1_*,\cos(2\pi (j)_*^1))$. 
Note by \eqref{Nsim} if $x_j^n(t)-j\leq \frac{1}{4}$, we get $\frac{d}{d t}x_j^n(t)\geq C j^{1-\alpha}$. Hence for any fix $t_0$ if $n$ is large enough, $x_j^n(t_0)-3j\geq \frac{1}{4}$. While $(x_j(t),y_j(t))$ goes to the opposite direction, so the distance between them is larger than some constant say $\frac{1}{4}$. And it cannot be too large since their limits are $((j+1)_*^1,\cos (j+1)_*^1),( (j)_*^1,\cos (j)_*^1))$. 

At last we choose $m_j=C_m e^{-j}$ where $C_m$ is a constant satisfying $C_m\sum_j e^{-j}=1$. Then it is straight forward to check that $\mu,\mu^n$ satisfy \textsc{(R)}. By direct computation we have
\[
\liminf_n d_W^2\left(\mu(t_0),\mu^n(t_0)\right)/d_W^2(\mu_0,\mu^n_0)\gtrsim 
\liminf_n(\sum_{j\geq n} C m_j)/(\sum_{j\geq n} m_j j^{-2\alpha})\geq \lim_n C n^{2\alpha}=+\infty.\]
We claim that this is the example promised which shows that $d_W(\mu(t),\mu^n(t))$ cannot be bounded by $C d_W(\mu_0,\mu^n_0)$ for $C$ independent of $t$ and $d_W(\mu_0,\mu^n_0)$.

Next we select $m_j=C_m' j^{-\beta}$ with $\beta>3$ where $C_m'$ is some constant such that the total mass of $\mu_0$ is $1$. In this case $\delta(R,\mu_0)\approx \sum_{j\geq R}j^{-\beta} j^{2}\approx R^{-\beta+1}$ which fails condition \textsc{(R)}. On the other hand
\[d_W^{2p}(\mu_0,\mu^n_0)\approx (\sum_{j\geq n} j^{-\beta}j^{-2\alpha})^p \approx n^{(-\beta-2\alpha+1)p} \text{ and }\]
\[d_W^{2}\left(\mu(t_0),\mu^n(t_0)\right)\gtrsim (\sum_{j\geq n} j^{-\beta}) \approx n^{-\beta+1}.\]
Thus for any $p> \frac{\beta-1}{\beta+2\alpha-1}$, we deduce
\begin{equation*}
\liminf_n d_W^2\left(\mu(t_0),\mu^n(t_0)\right) / d_W^{2p}(\mu_0,\mu^n_0)=+\infty \quad \text{ for all }t_0>0.
\end{equation*}
If selecting $\alpha=1-\epsilon,\beta=3-\epsilon$, then $p$ can be any close to $\frac{1}{2}$. This shows that in Theorem \ref{converge} if without \textsc{(R)}, $p$ cannot be greater than $\frac{1}{2}$. So at least we can claim that the stability is weaker. 
\end{proof}

\section{Part Two. Second Order Equations}\label{part2}

In the second part of this paper we show the well-posedness of the second order continuity equation \eqref{eqn2} and then we send the diffusion term to $0$. If $\Omega(t)$ is bounded and convex for each $t$, we will prove that \eqref{eqn2} is indeed the vanishing viscosity approximation of \eqref{eqn1}.

\subsection{Assumptions and JKO Scheme}\label{sub3.1}

We make the following assumptions. We will make a remark about several generalizations later.
 
\begin{flushleft}
\textbf{(O2)} The lateral boundary of $\Omega_T$ is uniformly $C^1$ in space. There exists $L>0$ that
 \begin{equation*}
d_H(\Omega(t),\Omega(s) )\leq L|t-s| \text{ for } 0\leq s,t\leq T.
\end{equation*}
\end{flushleft}
Here $d_H$ is the Hausdorff distance. From this we know that there exists $r_p>0$ such that both $\partial \Omega(t),\partial \Omega(t)^C$ are $r_p$-prox regular for all $t\in[0,T]$.
For $V,W$, we assume \textsc{(C1)-(C3)} hold and furthermore we assume: 
\begin{flushleft}
\textbf{(C4)} $V(\cdot),W(\cdot)\in C^2(\mathbb{R}^d)$ are bounded below. 
\end{flushleft}

Recall that the associated energy $\phi^\epsilon$ is defined in \eqref{phid}.
We define
\textit{the proper domain of functional $\phi^\epsilon$} is
\[Dom(\phi^\epsilon,t):=\left\{\mu\in \mathcal{P}_2(\overline{\Omega(t)}),\text{ } \phi^\epsilon(\mu)<+\infty\right\}.\]
Notice there is no difference between $\mu\in Dom(\phi^1,t)$ and $\mu\in Dom(\phi^\epsilon,t)$ for some $\epsilon>0$. Next
as a convention, 
\begin{equation*}
    \frac{\nabla u}{u}:=\begin{cases}
\frac{\nabla u}{u} &\text{ if }u\ne 0, \\
0 &\text{ if } \nabla u= 0,\\
+\infty &\text{ if } \nabla u\ne 0, u=0.
\end{cases}
\end{equation*}

\medskip

Without loss of generality, we only prove well-posedness for $\epsilon=1$. We have the following equation:
\begin{equation}\label{eqn2'}
\left\{\begin{aligned}
&\frac{\partial}{\partial t}\mu-\nabla\cdot ( \nabla \mu+\nabla V \mu+(\nabla W *\mu)\mu )= 0 & \text{ in }&\Omega_T,\\
&\left(\nabla \mu+\nabla V \mu+(\nabla W *\mu)\mu + c\mu\right)\cdot n=0 &
\text{ on }& \partial_l\Omega_T,\\
&\mu=\mu_0 & \text{ on }&  \Omega(0).
\end{aligned}
\right.
\end{equation}
Suppose $\mu_0 \in Dom(\phi^1,0)$ and conditions \textsc{(C1)-(C4)(O1)(O2)} hold.
We use the following variant of the celebrated JKO scheme. Fix a small time step $\tau>0$, define $J_{\tau,t}: \mathcal{P}_2^a\left(\Omega(t)\right)\rightarrow \mathcal{P}_2^a\left(\Omega(t+\tau)\right)$ by
\begin{equation}\label{JKOscheme}
J_{\tau,t}(\mu)\in argmin_{v\in \mathcal{P}_2\left(\Omega(t+\tau)\right)}\left\{\frac{1}{2\tau}d_W^2(\mu,v)+\phi^1(v)\right\}.    
\end{equation}
First we show the existence of such minimizers. 
With the assumptions (C1)-(C4) on $V,W$, we have $\phi^1$ is lower semi-continuous, coercive, compact. Then 
\[\inf_{v\in \mathcal{P}_2(\Omega(t+\tau))}\left\{\frac{1}{2\tau}d_W^2(\mu,v)+\phi^1(v)\right\}\] 
is bounded below. And we can find a sequence of measures whose energy converges to the infimum and they all belong to $\mathcal{P}_2^a$ due to the internal energy. Then lower semi-continuity of $\phi^1$ and compactness guarantee the existence of the limit. Details can be found in section 2.1 in \cite{gflow} or Lemma 4.2 of \cite{wuli}. Actually if $\{\Omega(t)\}$ is convex for all $t$, we have the uniqueness of the minimizer. However, here we only need the existence result.

\subsection{Some Estimates}

First we prove a technical lemma which is enlightened by Corollary 2.6 in \cite{sweeping}. It will be used to compare $\mu(t)$ with $J_{\tau,t}(\mu(t))$ whose support is different.  

\begin{proposition}\label{v vs u}
Suppose the domain satisfies condition \textsc{(O2)}. Then for $t\in [0,T-\tau]$ and $\mu\in\mathcal{P}_2^a\left(\Omega(t)\right)$, there exists a Lipschitz continuous map $\textbf{t}: \Omega(t)\rightarrow \Omega(t+\tau)$ such that
\begin{eqnarray}
&&\|\textbf{t}-\textbf{i}\|_{L^2(\Omega(t),\mu)}\leq CL\tau,\label{t i}\\
&&  \det (D\textbf{t})\geq 1-CL\tau \text{ for a.e. }x\in\Omega(t). \label{det D}
\end{eqnarray}
Here $C$ is some constant independent of $t,\tau$ which only depends on the geometry of $\Omega_T$. 
\end{proposition}

\begin{proof}
By \textsc{(O2)}, $\partial\Omega(t)$ is uniformly $C^1$ for all $t$. So there exists $r_p>0$ such that for each $t\in[0,T]$, there is a unique Lipschitz map $q_t: \partial \Omega(t)\rightarrow \Omega(t)$ such that
\[
q_t(y)=z \text{ with } d(z,y)=d\left(z,\partial \Omega(t)\right)=\frac{r_p}{2}.\]
Without loss of generality, we can assume $\partial\Omega(t)^{C}$ is $r_p$-prox regular. So for any $x$ near the boundary with distance $\leq \frac{r_p}{2}$ to the boundary, there exists a unique $y\in \partial \Omega(t)$ such that $d(x,y)=d\left(x,\partial \Omega(t)\right)$. We denote such $y=p_t(x)$. So \[p_t:\Omega^r_t:=\left\{x\in\overline{\Omega(t)}, d(x,\partial\Omega)\leq \frac{r_p}{2}\right\}\rightarrow \partial \Omega(t)\] and $q_t\circ p_t=id$ on $\partial\Omega^r_t/\partial \Omega(t)$.
Now define a continuous map $Q^t$: $\overline{\Omega(t)}\rightarrow \Omega(t)$,
\begin{equation*}
Q^t(x):=
\left\{\begin{aligned}
&q_t(p_t x) &\text{ if } d\left(x,\partial \Omega(t)\right)\leq \frac{r_p}{2},\\
&x & \text{ otherwise}.
\end{aligned}\right.
\end{equation*}
Note by the assumption made on the boundary of the domain
\[|p_t(x)-p_t(y)|\lesssim |x-y|,\quad |q_t(x)-q_t(y)|\lesssim |x-y|.\]
It is not hard to see that $DQ^t$ exists almost everywhere and $|DQ^t|$ is uniformly bounded.
Define
\[\textbf{t}(x)=(1-\frac{3L\tau}{r_p})x+\frac{3L\tau}{r_p}{Q^t(x)}.\] 
Then estimate \eqref{det D} is satisfied.

We want to show that $\textbf{t}: \Omega(t)\rightarrow \Omega(t+\tau)$. Otherwise if there exists $x\in\Omega(t)$ such that $\textbf{t}(x)\ne\Omega(t+\tau)$, by \textsc{(O2)} and from the geometry for any $z\in \partial B(0,\frac{r_p}{3})$, the line segment connecting $x$ and $Q^t(x)+z$ lies in $ \Omega(t)$. Then
\[\textbf{t}(x)+\frac{3L\tau}{r_p}z=(1-\frac{3L\tau}{r_p})x+\frac{3L\tau}{r_p}({Q^t(x)}+z)\in\Omega(t).\]
In view of the fact that $z$ is arbitrary with length  $\leq\frac{r_p}{3}$, we end up with a contradiction to the Lipschitz variation of $\Omega(\cdot)$.

Finaly for \eqref{t i},
\[ \int_{\Omega(t)}|\textbf{t}(x)-x|^2 d\mu=\left(\frac{L\tau}{r_p}\right)^2\int_{\Omega(t)}|-x+Q^t(x)|^2 d\mu
\leq C{L^2\tau^2}.\]
\end{proof}

As a corollary, we have the Lemma \ref{ti} below. Notice that the hypotheses \eqref{eqnti} is weaker than \eqref{t i} which is made to allow possible weaker assumptions (than (O2)) on the domain (see Remark \ref{corner}).

\begin{lemma}\label{ti}
Assume conditions (C1)(C2) hold and fix $t\geq 0$. Suppose for all $\tau>0$ small enough, $J_{\tau,t}$ is well-defined, and there is a map $\textbf{t}:\Omega(t)\rightarrow\Omega(t+\tau)$ and universal constants $C,q<\frac{1}{2}$ such that estimate \eqref{det D} holds and
\begin{equation}\label{eqnti}
\|\textbf{t}-\textbf{i}\|_{L^2(\Omega(t),\mu)}\leq CL\tau(1+m_2(\mu)^q).
\end{equation}
Then for some universal constants $q'<1,C'$ and any $\mu\in\mathcal{P}_2^a\left(\Omega(t)\right)$, we have
\[\frac{1}{2\tau}d_W^2\left(\mu,J_{\tau,t}(\mu)\right)+\phi^1\left(J_{\tau,t}(\mu)\right)
\leq \phi^1(\mu) +C'\tau(1+ m_2(\mu)^{q'}).\]
\end{lemma}
\begin{proof}
Write $v=\textbf{t}\# \mu$ and by the assumption 
\[d_W^2(\mu,v)\leq \int_{\Omega(t)}|\textbf{t}(x)-x|^2 d\mu\leq C\tau^2(1+m_2(\mu)^q)^2.\]
Then we estimate $\phi^1(v)-\phi^1(\mu)$. By simple calculation,
\[\textbf{t}\# \mu\circ \textbf{t}=\frac{\mu}{\det (D\textbf{t})}.\] Write ${U}(\mu)=u\log u$ and then
\begin{align*}
&    \int_{\Omega(t+\tau)} {U}(v)dx-\int_{\Omega(t)} {U}(\mu)dx=\int_{\Omega(t+\tau)} \frac{{U}(\textbf{t}\# \mu)}{\textbf{t}\# u}d\textbf{t}\# \mu-\int_{\Omega(t)} \frac{{U}(\mu)}{u}d\mu\\
=&\int_{\Omega(t)} (\frac{{U}(\textbf{t}\# \mu\circ \textbf{t})}{\textbf{t}\# u\circ \textbf{t}}-\frac{{U}(\mu)}{u})d\mu=\int_{\Omega(t)} \left({U}(\frac{\mu}{\det (D\textbf{t})})\det (D\textbf{t})-{U}(\mu)\right)dx\\
\leq&\int_{\Omega(t)\cap \{\det (D\textbf{t})\leq 1\}} -u\log \det D\textbf{t} dx\leq C\tau.
\end{align*}

Next we compare $\mathcal{V}(v)$ with $\mathcal{V}(\mu)$.
From \textsc{(C1)}, $|\nabla V(x)|^2\leq C(1+|x|^2)$. Then
\begin{align*}
    &\int_{\Omega(t+\tau)} Vdv-\int_{\Omega(t)} Vd\mu=\int_{\Omega(t)} V\left(\textbf{t}(x)\right)-V(x)d\mu\\
    = &\int_{\Omega(t)}\nabla V(\theta)\cdot(\textbf{t}(x)-x)d\mu\leq C\l\int_{\Omega(t)}(1+|x|^2+|\textbf{t}(x)|^2)d\mu \r^{\frac{1}{2}}\l\int_{\Omega(t)}|\textbf{t}(x)-x|^2d\mu\r^{\frac{1}{2}}\\
    \leq & C\tau(1+m_2(\mu)^\frac{1}{2}+\tau m_2(\mu)^q)(1+m_2(\mu)^q)\leq  C\tau(1+m_2(\mu)^{q'}).
\end{align*}
Here the $\theta$ lies in the segment connecting $x,\textbf{t}(x)$ by mean-value theorem. And $q'=\frac{1}{2}+q<1$.
The last but two inequality holds because:
$|\textbf{t}(x)|2\leq 2|x|^2+2|\textbf{t}(x)-x|^2$ and $\|\textbf{t}-\textbf{i}\|_{L^2(\Omega(t),u)}$ can be bounded by $C\tau(1+m_2(\mu)^q)$ by the assumption. The last one by H\"{o}lder inequality and boundedness of $m_2(\mu)$.

Similar computation yields
\begin{align*}
    &\int_{\Omega(t+\tau)\times\Omega(t+\tau)} W(x-y)dv(y)d v(x)-\int_{\Omega(t)\times \Omega(t)} W(x-y)d\mu(y)d \mu(x)
\\
\leq &C\l\int_{\Omega(t)}(1+|x|^2+|\textbf{t}(x)|^2+|y|^2+|\textbf{t}(y)|^2)d\mu^2 \r^{\frac{1}{2}}\l\int_{\Omega(t)}|\textbf{t}(x)-x|^2+|\textbf{t}(y)-y|^2d\mu^2\r^{\frac{1}{2}}\\
\leq & C\tau(1+m_2(\mu)^{q'}).
\end{align*}
In all we proved
\[\phi^1(v)\leq \phi^1(\mu) + C\tau(1+m_2(\mu)^{q'}).\]
Then the optimality of $J_{\tau,t}(\mu)$ gives:
\[\frac{1}{2\tau}d_W^2\left(\mu,J_{\tau,t}(\mu)\right)+\phi^1\left(J_{\tau,t}(\mu)\right)\leq \frac{1}{2\tau}d_W^2(\mu,v)+\phi^1(v)\]\[
\leq \phi^1(\mu) +C\tau(1+ m_2(\mu)^{q'})+C\tau(1+m_2(\mu)^q)^2.\]
Note $2q<q'$, we finished the proof.
\end{proof}

Define \[\mu^k_\tau:=J_{\tau,(k-1)\tau}\circ ...\circ J_{\tau,0}(\mu_0)\in\mathcal{P}(\overline{\Omega(k\tau)}).\] 
We want to apply Lemma \ref{ti} on every $\mu_\tau^k$.

\begin{proposition}\label{prop 2}
Assume (C4) and under the assumption of Lemma \ref{ti}. For fixed $\mu_0\in Dom(\phi^1,0), T$, if $\tau$ is small enough and $n\tau<T$, then there exists $C>0$ independent of $\tau,k,n$ such that
\[\sum_{k=0}^{n-1} d_W^2(\mu_\tau^k,\mu_\tau^{k+1})
\leq C\tau,\quad \phi^1(\mu_\tau^n)\leq C.\]
\end{proposition}

\begin{proof}
By Lemma \ref{ti},
\begin{equation*}
\frac{1}{2\tau}d_W^2(\mu^{k}_\tau,\mu^{k+1}_\tau)+\phi^1(\mu^{k+1}_\tau)\leq \phi^1(\mu^{k}_\tau) +C\tau(1+ m_2(\mu^{k}_\tau)^{q'}). 
\end{equation*}
By iteration
\begin{equation}
\frac{1}{2\tau}\sum_{k=0}^{n-1} d^2_W(\mu_\tau^k,\mu_\tau^{k+1})+\phi^1(\mu^{n}_\tau)\leq
\phi^1(\mu_0) +C\tau n+C\tau m_2(u_\tau^0)^{q'}+...
+C\tau m_2(\mu^{n-1}_\tau)^{q'}. \label{tau energy}
\end{equation}
Note $d_W(\mu_\tau^0,\mu_\tau^n)\leq \sum_{k=0}^{n-1} d_W(\mu_\tau^k,\mu_\tau^{k+1})$, we obtain
\begin{equation} \label{ineq dist}
\begin{aligned}
\frac{1}{2n\tau}d_W^2(\mu_\tau^0,\mu_\tau^n)&\leq \frac{1}{2\tau}\sum_{k=0}^{n-1} d_W^2(\mu_\tau^k,\mu_\tau^{k+1})\\
&\leq \phi^1(\mu_0)-\phi^1(\mu_\tau^n)+C\tau \sum_{k=1}^{n-1} m_2(\mu_\tau^k)^{q'}+C.
\end{aligned}
\end{equation}
To give a bound to $m_2(\mu)$, we use the trick as in proposition 4.1 \cite{JKO}. 
Since 
\[
m_2(\mu_\tau^n)\leq 2m_2(\mu_\tau^0)+2d_W^2(\mu_\tau^0,\mu_\tau^n),
\]
and by \eqref{ineq dist} and the lower bound of $V,W$, we see
\[
m_2(\mu_\tau^n)
\leq C+C\tau \sum_{k=0}^{n-1} m_2(\mu_\tau^k)^{q'}.\]
Here ${q'}<1$. Considering that $C$ is independent of $\tau,n$ (which only depends on $T, \mu_0, \Omega_T$) and $n$ can be any positive integer such that $n\tau <T$, the above shows $m_2(\mu_\tau^n)$ is uniformly bounded.

According to \eqref{ineq dist}
\[\sum_{k=0}^{n-1} d_W^2(\mu_\tau^k,\mu_\tau^{k+1})
\leq C\tau.\]
And in view of \eqref{tau energy}, $\{\phi^1(\mu_\tau^k)\}$ are uniformly bounded.
\end{proof}

\subsection{Convergence of Discrete Solutions}

We define a discrete type solution with time step $\tau$ as
\begin{equation}\mu_\tau(t):=J_{\tau,(k-1)\tau}\circ ...\circ J_{\tau,0} (\mu_0)=\mu^k_\tau \quad\text{if  }t\in((k-1)\tau,k\tau].\label{jkoscheme}
\end{equation}
As proved in \cite{JKO},
\[\left\{\mu\in\mathcal{P}_2( \mathbb{R}^d ):\phi^1(\mu)\leq C, m_2(\mu)\leq C\text{ for some }t\leq T\right\}\] is compact in $\mathcal{P}^a_2( \mathbb{R}^d )$. 

Then according to Proposition \ref{prop 2}, $\{\mu_\tau(t), t\leq T, \tau>0\}$ is a compact subset in $\mathcal{P}_2^a( \mathbb{R}^d )$. We connect every pair of consecutive discrete values $(\mu^{k-1}_\tau, \mu^{k}_\tau)$ with a constant speed geodesic parametrized in each interval $[(k-1)\tau,k\tau]$ by
\[\hat{\mu}_\tau\left((k-1)\tau+s\right):=\left((1-s)\textbf{i}+s\textbf{t}\right)_\# \mu^{k-1}_\tau, s\in[0,1].\]
Here $\textbf{t}$ is an optimal transport map from $\mu^{k-1}_\tau$ to $\mu^{k}_\tau$. Again by Proposition \ref{prop 2}, $\{\hat{\mu}_\tau, \tau>0\}$ are H\"{o}lder continuous curves. Ascoli-Arzela Theorem yields the relative compactness of $\hat{\mu}_\tau$ in 
$C^0\left([0,T];\mathcal{P}_2( \mathbb{R}^d )\right).$
Then there is a subsequence of $\hat{\mu}_\tau(\cdot)$ that converge to some $\mu(\cdot)$ in Wasserstein metric uniformly pointwise. Obviously $\mu(t)$ is supported in $\Omega(t)$ and $\mu_\tau(\cdot)\rightarrow \mu(\cdot)$ along the subsequence $\tau\rightarrow 0$. We write $u(t)$ as $\mu(t)$'s density function.

From a by-now standard computation presented in \cite{JKO,la}, the Euler-Lagrange equation associated with \eqref{JKOscheme} is as follows: for any $\psi$, a smooth vector field compactly supported in $\Omega(k\tau)$,
\begin{equation}\label{EL}
\begin{aligned}
&\quad\quad\int_{\Omega((k-1)\tau)\times\Omega(k\tau)}(y-x)\cdot \psi(y)d\gamma^k_\tau+\\
&\tau\int_{\Omega(k\tau)}(-\nabla \cdot \psi+\nabla V\cdot \psi +\nabla W*\mu_\tau^k \cdot \psi )(y)d\mu_\tau^k(y)=0.
\end{aligned}
\end{equation}
Here $\gamma^k_\tau=\gamma^k_\tau(x,y)$ is an optimal transport plan between $\mu^{k-1}_\tau(x)$ and $\mu^{k}_\tau(y)$. 

Suppose $\psi$ is a smooth compactly supported vector field. We mainly apply H\"{o}lder's inequality and conditions (C1)(C2) to get the following estimates.
\begin{align*}
&\left|\int_{\Omega((k-1)\tau)\times\Omega(k\tau)}(y-x)\cdot \psi(y)d\gamma^k_\tau\right|\leq \left(\int_{\Omega((k-1)\tau)\times\Omega(k\tau)}|y-x|^2 d\gamma^k_\tau\right)^\frac{1}{2}\left(\int_{\Omega(k\tau)} |\psi|^2d\mu_\tau^k\right)^\frac{1}{2}\\
&\quad\leq \|\psi\|_{L^2(\mathbb{R}^d,d\mu_\tau^k)}d_W(\mu_\tau^{k-1},\mu_\tau^k),\\
&\left|\int_{\Omega(k\tau)} \nabla V\cdot \psi\, u_\tau^k dy\right|\leq \|\psi\|_{L^2(\mathbb{R}^d,d\mu_\tau^k)}\left(\int_{\Omega(k\tau)\cap B^c_R} |\nabla V|^2 d\mu_\tau^k\right)^\frac{1}{2}\leq C\|\psi\|_{L^2(\mathbb{R}^d,d\mu_\tau^k)}(m_2^\frac{1}{2}\l\mu_\tau^k)+1\r ,\\
&\left|\int_{\Omega(k\tau)} \nabla W*\mu^k_\tau\cdot \psi\, u_\tau^k dy \right|\leq C\|\psi\|_{L^2(\mathbb{R}^d,d\mu_\tau^k)}\left(\int_{\Omega^2(k\tau)\cap\{|y|\geq R\}} 1+|y-z|^2d\mu_\tau^k(z) d\mu_\tau^k(y)\right)^\frac{1}{2}\\
&\quad\leq C\|\psi\|_{L^2(\mathbb{R}^d,d\mu_\tau^k)}(m_2^\frac{1}{2}\l\mu_\tau^k)+1\r.
\end{align*}
By the uniform bound of $m_2(\mu^k_\tau)$, we find
\begin{equation}
    \label{utaok L2}
    \left|\int_{\Omega(k\tau)}u_\tau^k\nabla\cdot \psi dx\right|\leq \frac{1}{\tau}d_W(\mu_\tau^k,\mu_\tau^{k-1})\|\psi\|_{L^2(\mathbb{R}^d,d\mu_\tau^k)}+C\|\psi\|_{L^2(\mathbb{R}^d,d\mu_\tau^k)}.
\end{equation}

Now we assume that $\psi$ is $0$ outside of $B_R$, which is an empty set if $R=0$. Denote 
\[\delta^k_{\tau,R}:=\int_{B^c_R}d\mu_\tau^k.\]
By the uniform boundedness of the second moment, $\delta^k_{\tau,R}\to 0$ as $R\to\infty$
uniformly in $\tau,k$ with $k\tau \leq T$. Then
\begin{align*}
&\left|\int_{\Omega((k-1)\tau)\times\Omega(k\tau)}(y-x)\cdot \psi(y)d\gamma^k_\tau\right|\leq \|\psi\|_{L^\infty}d_W(\mu_\tau^{k-1},\mu_\tau^k)\left| \delta^k_{\tau,R}\right|^\frac{1}{2},\\
&\left|\int_{\Omega(k\tau)} \nabla V\cdot \psi u_\tau^k dy\right|\leq \|\psi\|_{L^\infty}\int_{\Omega(k\tau)\cap B^c_R} |\nabla V| d\mu_\tau^k\leq C\|\psi\|_{L^\infty}(m_2^\frac{1}{2}\l\mu_\tau^k)+1\r \left| \delta^k_{\tau,R}\right|^\frac{1}{2},\\
&\left|\int_{\Omega(k\tau)} \nabla W*\mu^k_\tau\cdot \psi u_\tau^k dy\right|\leq C\|\psi\|_{L^\infty}\int_{\Omega^2(k\tau)\cap\{|y|\geq R\}} 1+|y-z|\,d\mu_\tau^k(z) d\mu_\tau^k(y)\\
&\quad\leq C\|\psi\|_{L^\infty}
\left(\int_{\{|y|\geq R\}} 1+|y|^2+|z|^2d\mu_\tau^k(z) d\mu_\tau^k(y)\right)^\frac{1}{2}
\left(\int_{\{|y|\geq R\}} d\mu_\tau^k(z) d\mu_\tau^k(y)\right)^\frac{1}{2}\\
&\quad\leq C\|\psi\|_{L^\infty}(m_2^\frac{1}{2}\l\mu_\tau^k)+1\r \left| \delta^k_{\tau,R}\right|^\frac{1}{2}.
\end{align*}
We find
\begin{equation}
    \label{utaok tightness}
    \left|\int_{\Omega(k\tau)}u_\tau^k\nabla\cdot \psi dy\right|\leq \frac{1}{\tau}d_W(\mu_\tau^k,\mu_\tau^{k-1})\|\psi\|_{L^\infty}\left| \delta^k_{\tau,R}\right|^\frac{1}{2}+C\|\psi\|_{L^\infty}\left| \delta^k_{\tau,R}\right|^\frac{1}{2}.
\end{equation}
Then we have
\[\|\nabla u_\tau^k\|_{L^{\infty,*}(\Omega(k\tau),dx)} \leq \frac{C}{\tau}d_W(\mu_\tau^{k-1},\mu_\tau^k)+C\]
which shows that $\nabla u_\tau$ exists in the dual space of $L^{\infty}(\mathbb{R}^d)$ denoted as $L^{\infty,*}(\mathbb{R}^d)$.
Actually applying Proposition \ref{prop 2} gives, $\nabla u_\tau\chi_{\{u_\tau>0\}}$ is uniformly bounded in $L^{\infty,*}(\mathbb{R}^d\times [0,T])$ for all $\tau>0$.
Also 
tightness is guaranteed by \eqref{utaok tightness}. So as $\tau\to 0$ along a subsequence
$\nabla u_\tau\to \nabla u $
weakly in $L^{\infty,*}(\Omega_T)$.

Now we take $\tau\to 0$ in \eqref{utaok L2} to find that 
\[\nabla u/ u\in L^2(\mathbb{R}^d\times [0,T], u\,dx dt).\]
Thus $\nabla u\,\chi_{\{u>0\}} \in L^1(\mathbb{R}^d\times [0,T])$ and
\begin{equation}
    \label{nablau converge}
    \nabla u_\tau\to \nabla u \quad\text{ in }L^1(\Omega_T)
\end{equation}

\medskip

By \eqref{EL} and approximations, for $\psi\in C^\infty_0(\mathbb{R}^d;\mathbb{R}^d)$ 
\begin{equation}\label{EL'}
\begin{aligned}
&\quad\quad\int_{\Omega((k-1)\tau)\times\Omega(k\tau)}(y-x)\cdot \psi(y)d\gamma^k_\tau+\\
&\tau\int_{\Omega(k\tau)}\left(\nabla V\cdot \psi +\nabla W*\mu_\tau^k \cdot \psi \right)d\mu_\tau^k+\tau\int_{\Omega(k\tau)}\psi\cdot\nabla u_\tau^k\,dy=0.
\end{aligned}
\end{equation}
Write $\textbf{t}^k_\tau$ as an optimal transport map from $\mu^{k}_\tau$ to $\mu^{k-1}_\tau$.
Then
\begin{equation}\label{psiktau}
\int_{\Omega((k-1)\tau)\times\Omega(k\tau)}(y-x)\cdot \psi(y)d\gamma^k_\tau=\int_{\Omega(k\tau)}(\textbf{t}^k_\tau-\textbf{i})\psi\left(\textbf{t}^k_\tau(x)\right) d\mu^k_\tau.
\end{equation}
For every test function $\varphi\in C_c^\infty\left(\mathbb{R}^d\times (0,T)\right)$
\begin{equation}\label{phiktau}
\left|\int_{\Omega(k\tau)}\varphi\left(\textbf{t}^k_\tau(x)\right)-\varphi(x)
-\langle\nabla\varphi(\textbf{t}^k_\tau(x)),\textbf{t}^k_\tau-\textbf{i}\rangle d\mu^k_\tau\right|\leq \frac{1}{2}\|\nabla^2 \varphi\|_{L^\infty}d_W^2(\mu^{k-1}_\tau,\mu^k_\tau).
\end{equation}

Take $\psi=\nabla \varphi$. By \eqref{nablau converge} \eqref{EL'} \eqref{psiktau} \eqref{phiktau} and  Proposition \ref{prop 2}, we get
\begin{align*}
    \int_{\Omega_T}\partial_t \varphi d\mu &=\lim_{\tau\rightarrow 0}\int_{\Omega_T}\partial_t \varphi d\mu_\tau=\lim_{\tau\rightarrow 0}\frac{1}{\tau}\sum_k\int_{\Omega(k\tau)}\varphi(x)-\varphi\left(\textbf{t}^k_\tau(x)\right)d\mu_\tau^k
\\
&=\lim_{\tau\rightarrow 0}\sum_k\int_{\Omega(k\tau)}(\nabla V\cdot \nabla \varphi +\nabla W*\mu_\tau^k \cdot \nabla \varphi )d\mu_\tau^k
+\int_{\Omega(k\tau)}\nabla \varphi\cdot \nabla u_\tau^k \,dxdt
\\
&=\int_{\Omega_T}(\nabla V\cdot \nabla \varphi +\nabla W*\mu \cdot \nabla \varphi )d\mu dt+\int_{\Omega_T}\nabla\varphi \cdot\nabla u\,dxdt.
\end{align*}
Till now we proved that $\mu$ is a weak solution to equation \eqref{eqn2'}.
We conclude with the following theorem.
\begin{theorem}\label{existence}
Suppose \textsc{(C1)-(C4)(O2)} hold. Then for $\mu_0\in Dom(\phi^1,0)$, there exists an absolutely continuous curve $\mu(\cdot)$ in $\mathcal{P}_a^2(\mathbb{R}^d)$ which solves equation \eqref{eqn2'} weakly in $\Omega_T$. 
\end{theorem}
\begin{proof}
From the above discussion, along a subsequence of $\tau\rightarrow 0$, $\mu^\tau(\cdot)$ converges narrowly to $\mu(\cdot)$ uniformly for all $t\in[0,T]$. The limit $\mu(\cdot)$ is an absolutely continuous curve in $\mathcal{P}_a^2(\mathbb{R}^d)$ and it is a weak solution to equation \eqref{eqn2'}. 
\end{proof}

\begin{remark}
In stationary domain, the rate at which solutions to the discrete gradient flow converge to solutions of the gradient flow was studied in \cite{gflow, katy}.  But we cannot prove the exponential formula in the case that the domain is time-dependent. This can be an interesting direction for future research.
\end{remark}

\begin{remark}\label{corner}
Here we required \textsc{(O2)} on the domain. However uniform $C^1$ regularity is only used in Proposition \ref{v vs u} and so that we can apply Lemma \ref{ti}. But according to the assumptions made in the lemma, \textsc{(O2)} is more than what is needed. For example, bound \eqref{eqnti} can still be achieved if there is a wedge on the boundary. It is technical to construct the map $\textbf{t}$ which depends on the geometry of the time-dependent domain.
\end{remark}

\subsection{Uniqueness Result}

We state two uniqueness results of equation \eqref{eqn2'}. In the first one, we study the $L^2$ stability of solutions in a stationary domain. 
The proof is postponed to the appendix. The second one is stated in the remark below where we require the space-time domain to be convex.

\begin{theorem}\label{uniqueness}
Suppose the domain $\Omega\subset\mathbb{R}^{d}$ is stationary and bounded with \textsc{(C1)-(C3)(O2)} hold. Suppose $\mu_0\in\mathcal{P}_2^a(\Omega)$ and its density $u_0\in L^2(\Omega)$. Then there exists a unique weak solution $\mu(\cdot)$ to equation \eqref{eqn2'} with density $u(t)\in L^2(\Omega)$ for each $t$.

If $\mu^i(t) (i=1,2)$ are two solutions with initial data $\mu_0^i\in\mathcal{P}_2^a(\Omega)$ with their densities $u_0^i\in L^2(\Omega)$, then there is $C$ depending on the domain and universal constants such that for a.e. $t\in[0,T]$
\begin{equation*}
\|u^1(t)-u^2(t)\|_{L^2(\Omega)}\leq C\|u^1_0-u^2_0\|_{L^2(\Omega)}.
\end{equation*}
\end{theorem}

\begin{remark}\label{remarkunique}
Assuming $\Omega(t)$ is uniformly bounded and convex for all $t\leq T$, then if $\mu^i(t) (i=1,2)$ are two solutions with initial data $\mu_0^i\in\mathcal{P}_2^a(\Omega(0))$, there is a universal $C(T)$ such that
\begin{equation*}
d_W(u^1(t),u^2(t))\leq Cd_W(u^1_0,u^2_0) \quad \text{ for }t\in [0,T].
\end{equation*}
This claim can be proved in a similar way as done in Theorem 11.1.4 \cite{gflow} or section 2 of \cite{unique}, as long as the domain is convex at any fixed time.

\end{remark}

\subsection{Convergence to the First Order Equation}

We consider equations \eqref{eqn1} and \eqref{eqn2} in bounded, convex domain in this section. Let $\mu^\epsilon$ be the weak solution to \eqref{eqn2} and $\mu$ be the weak solution to \eqref{eqn1}.
We want to show that $\mu^\epsilon$ converges to $\mu$ in Wasserstein metric as $\epsilon\rightarrow 0$.

Recall \eqref{phid} and write $\phi(\mu), \phi^\epsilon(\mu^\epsilon)$ as the energies.
\textit{The internal energy} is denoted as
\[\mathcal{U}^\epsilon(\mu)=\epsilon\int_{\mathbb{R}^d}u\log u dx \text{ where }\mu=u \mathcal{L}^d.\]
\textit{The metric slope of functional $\phi^\epsilon$} for $\mu\in Dom(\phi^\epsilon,t)$ at time $t$ is 
\[|\partial\phi^\epsilon(t)|(\mu):=\limsup_{w\rightarrow \mu, w\in \mathcal{P}_2^a\left(\Omega(t)\right)}\frac{\left(\phi^\epsilon(\mu)-\phi^\epsilon(w)\right)^+}{d_W(w,\mu)}.\]

Now we give two lemmas. The proof of the first lemma is standard (see Proposition 10.4.13 \cite{gflow}), but we still need to be careful since the domain is time-dependent. We postpone the proof in the appendix.

\begin{lemma}\label{slope}
Suppose \textsc{(C1)-(C4)} hold, the domain is bounded and satisfies conditions \textsc{(O2)}, and $\mu_0\in Dom(\phi^1,0)$. Let $v^\epsilon(x,t)=\l\epsilon\frac{\nabla  u^\epsilon}{u^\epsilon}+\nabla V+\nabla W*\mu^\epsilon \r(x,t)$. Then for any $0<\epsilon<1$, \[\int_{\Omega_T}\left|v^\epsilon(x,t)\right|^2  u^\epsilon(x,t) dxdt \leq C .\]
\end{lemma}

\begin{cor}\label{3.3}
Settings are as above. For any $0<T'\leq T$, $\int_{0\leq t\leq T'}\mathcal{U}^\epsilon(\mu^\epsilon(t))dt \rightarrow 0$ as $\epsilon\rightarrow 0$.
\end{cor}
\begin{proof}
By the Euclidean Logarithmic Sobolev inequality (see \cite{sobolev}\cite{sobolev1}) and the fact that $u^\epsilon(t)$ is supported in $\Omega(t)$,
\[\int_{\Omega(t)}u^\epsilon\log u^\epsilon dx \leq \frac{d}{2}\log (\frac{1}{2\pi d e}\int_{\Omega(t)}\frac{|\nabla u^\epsilon|^2}{u^\epsilon}dx).\]
Write $\mathcal{U}^\epsilon(t):=\mathcal{U}^\epsilon(\mu^\epsilon(t))$. Then
\[\int_0^T\exp\left(\epsilon^{-1}\mathcal{U}^\epsilon(t)\right)dt\leq C\int_{\Omega_T}\frac{|\nabla u^\epsilon|^2}{u^\epsilon}dxdt\leq C\epsilon^{-2}. \]
We used Lemma \ref{slope} and the regularity of $V,W$ in the last inequality. Now for $\epsilon$ small enough, assume $\epsilon^2 e^N\geq N$ holds for all $N\geq \epsilon^{-\frac{1}{2}}$. Thus
\[\int_0^T \epsilon^{-1}\mathcal{U}^\epsilon(t)dt\leq \int_0^T \epsilon^{-\frac{1}{2}}dt+\int_{\mathcal{U}^\epsilon(t)\geq \epsilon^{\frac{1}{2}}} \epsilon^{-1}\mathcal{U}^\epsilon(t)dt\]
\[\leq C\epsilon^{-\frac{1}{2}}+\int_0^T \epsilon^2\exp\left(\epsilon^{-1}\mathcal{U}^\epsilon(t)\right)dt\leq C\epsilon^{-\frac{1}{2}}\]
which finishes the proof.
\end{proof}

The following lemma is one important ingredient to the proof of the convergence. Note it is possible that $\mu\notin$ the proper domain of $\phi^1$ (or equivalently of $\phi^\epsilon$), the plan is to regularize it and replace it by a $\tilde{\mu}\in\mathcal{P}^2_a$. As explained in the introduction, we look for a $\tilde{\mu}$ with density function uniformly bounded by $\epsilon^{-\alpha}$ for some $0<\alpha<1$. 
Additionally we need $d_{\mu^\epsilon}(\mu,\tilde{\mu})$ to be small where $d_{\mu^\epsilon}(\cdot,\cdot)$ is the Pseudo-Wasserstein metric with base ${\mu^\epsilon}$. As a remark, this is stronger than requiring $d_W(\mu,\tilde{\mu})$ to be small.

\begin{lemma}\label{modi}
Given any $\mu\in\mathcal{P}_2(\overline{\Omega}),v\in\mathcal{P}^a_2({\Omega})$ where $\Omega$ is a bounded, convex subset of $\mathbb{R}^d$. For any $s>0$ small enough, there exists
$\mu_s\in\mathcal{P}^2_a(\Omega)$ such that 
\begin{eqnarray*}
&d_v(\mu,\mu_s)\leq Cs \text{ and }\\ &\max\left\{\mu_s(x),x\in\Omega\right\}\leq s^{-d}.
\end{eqnarray*}
The constant $C$ only depends on the diameter and the volume of $\Omega$.
\end{lemma}
\begin{proof}

Without loss of generality, suppose $\Omega$ has volume $1$ in Euclidean measure. Let $e$ be the Euclidean measure restricted in $\Omega$ and then $e\in\mathcal{P}^2_a(\Omega)$.  Since $v$ is absolutely continuous, $\textbf{t}_v^e$ and $\textbf{t}_v^{\mu}$ exist and $\textbf{t}_v^e$ is one to one on $\Omega$ outside a $v$ zero measure subset. Let
\[\mu_s:=\left(\left((1-s)\textbf{t}^{\mu}_{v}+s\textbf{t}^{e}_{v}\right)_\# v\right)\] be the generalized geodesic joining ${\mu},e$ with base $v$, which is defined as in Definition 9.2.2 \cite{gflow}. Due to the convexity of the domain, we have $\mu_s\in\mathcal{P}_2(\Omega)$. By Proposition 2.6.4 \cite{katy}, the generalized geodesic is of constant speed in the sense that
\[ d_v({\mu},\mu_s)=s d_v({\mu},e).\]
Since the domain is bounded, $d_v({\mu},e)$ is uniformly bounded for all probability measures $v,{\mu},e$. We deduce that
$d_v({\mu},\mu_s)\leq Cs$.

Now we show the pointwise boundedness of $\mu_s$. Let $\varphi=\chi_{B_r}(x)$ which equals $1$ in $B_r$ and $0$ outside. Thus
\begin{equation}\label{Bxr}
\int_\Omega \varphi d\mu_s=\int_\Omega \varphi\left((1-s)\textbf{t}^{\mu}_{v}+s\textbf{t}^{e}_{v}\right)dv=v\left\{\left((1-s)\textbf{t}^{\mu}_{v}+s\textbf{t}^{e}_{v}\right)^{-1}B_r(x)\right\}    
\end{equation}
Write $S:=\left((1-s)\textbf{t}^{\mu}_{v}+s\textbf{t}^{e}_{v}\right)^{-1}B_r(x)$. 
By definition
\[vol\left\{B_r(x)\right\}=vol\left\{\left((1-s)\textbf{t}^{\mu}_{v}+s\textbf{t}^{e}_{v}\right)S\right\}.\]
Now we apply  Brunn-Minkowski inequality (Lemma \ref{minkowski}) to find the above
\[\geq vol\left\{s\textbf{t}^{e}_{v}S\right\}=s^dv(S).\]

So
\[v(S)\leq s^{-d} vol\{B_r(x)\}= vol\left\{B_0(1)\right\}(\frac{r}{s})^d.\] By \eqref{Bxr}, for any $\varphi=\chi_{B_r}(x)$ we find out
\[\frac{1}{vol\left\{B_r(x)\right\}}\int_\Omega \varphi d\mu_s\leq s^{-d}.\]
This shows that $u_s$ is an $L^\infty$ function in $\Omega$ with bound $s^{-d}$.
\end{proof}

Now we give our main theorem in the second part of this paper.

\begin{theorem}\label{convergence}
Assume \textsc{(C1)-(C4)(O1)(O2)} hold and $\mu_0\in Dom(\phi^1,0)$.
Suppose $\Omega_T$ is bounded and $\Omega(t)$ is convex for all $t$. Let $\mu^\epsilon$ be the unique weak solutions to equations \eqref{eqn2} and $\mu$ be the unique weak solution to equation \eqref{eqn1}. Then there exist some constants $c,C$ depending on the universal constants and the domain such that
\[d^2_W(\mu,\mu^\epsilon)(t)\leq C\epsilon^\frac{1}{d+2}t e^{ct} \text{ for all } t\in[0,T].\]
\end{theorem}
\begin{proof}

For any ${\omega_1}\in\mathcal{P}^a_2\left(\Omega(t)\right)$, let $\mu_s:=(s\textbf{t}^{\omega_1}_{\mu^\epsilon}+(1-s)\textbf{i})_\# \mu^\epsilon$ with $\mu^\epsilon=\mu^\epsilon(t)$. The convexity of the domain implies $\mu_s\in \mathcal{P}_2^a(\Omega(t))$. For any Fr\'echet subdifferential of $\phi^\epsilon$ at $\mu^\epsilon$ (see section 10 \cite{gflow}) $\xi^\epsilon\in L^2(\mu^\epsilon;\mathbb{R}^d)$, we have
\[\liminf_{s\rightarrow 0}\frac{\phi^\epsilon(\mu_s)-\phi^\epsilon(\mu^\epsilon)}{s}\geq \int_{\Omega(t)}\langle\xi^\epsilon,\textbf{t}^{\omega_1}_{\mu^\epsilon}-\text{i}\rangle d\mu^\epsilon.\] By \textsc{(C3)}, $\phi^\epsilon$ is $\tilde{\lambda}-$convex for $\tilde{\lambda} = \min\{\lambda, 3\lambda\}$. So by the Characterization by Variational inequalities and monotonicity in 10.1.1  \cite{gflow}, 
\[\frac{\phi^\epsilon(\mu_s)-\phi^\epsilon(\mu^\epsilon)}{s}\leq \phi^\epsilon({\omega_1})-\phi^\epsilon(\mu^\epsilon)-\frac{\tilde{\lambda}}{2}(1-s)d_W^2({\omega_1},\mu^\epsilon).\]
Then we take $s\rightarrow 0$ and find \begin{equation}\label{phiexi}
    \phi^\epsilon({\omega_1})-\phi^\epsilon(\mu^\epsilon)\geq \int_{\Omega(t)}\langle\xi^\epsilon,\textbf{t}^{{\omega_1}}_{\mu^\epsilon}-\textbf{i}\rangle d\mu^\epsilon+\frac{\tilde{\lambda}}{2}d_W^2({\omega_1},\mu^\epsilon).
\end{equation}
By the JKO scheme, $\mu^\epsilon$ is a gradient flow solution and we can choose $\xi^\epsilon=-v^\epsilon$, the tangent velocity field of $\mu^\epsilon$.

Similarly since $\mu$ is a gradient flow solution, $\xi:=P_{x,t}(-\nabla V-\nabla W*\mu)=-v$ is one Fr\'echet subdifferential of $\phi$ at $\mu$ and then for any ${\omega_2}\in\mathcal{P}_2(\overline{\Omega(t)})$
\begin{equation}\label{phixi}
    \phi({\omega_2})-\phi(\mu)\geq \int_{\Omega(t)}\langle\xi,\textbf{t}^{{\omega_2}}_{\mu}-\textbf{i}\rangle d\mu+\frac{\tilde{\lambda}}{2}d_W^2({\omega_2},\mu).
\end{equation}

For each $t$ we use Lemma \ref{modi} to modify $\mu$. Take $v=\mu^\epsilon,s=\epsilon^\frac{1}{d+2}$ and let $\tilde{\mu}=\mu_{s}\in \mathcal{P}^a_2(\Omega(t))$ with $\tilde{\mu}=\tilde{u}\mathcal{L}^d$. Then for all $0\leq t\leq T$
\[\max \left\{\tilde{u}(x,t)\right\}\leq \epsilon^{-\frac{d}{d+2}}, \quad d_{\mu^\epsilon}(\tilde{\mu},\mu)\leq C\epsilon^\frac{1}{d+2}.\]
Plug in $\omega_1=\tilde{\mu}$ in \eqref{phiexi}, 
\[\phi^\epsilon\left(\tilde{\mu}(t)\right)-\phi^\epsilon\left(\mu^\epsilon(t)\right)\geq \int_{\Omega(t)}\langle\xi^\epsilon,\textbf{t}^{\tilde{\mu}}_{\mu^\epsilon}-\textbf{i}\rangle d\mu^\epsilon+\frac{\tilde{\lambda}}{2}d_W^2(\tilde{\mu},\mu^\epsilon)\]
\begin{equation*}
 \geq \int_{\Omega(t)}\langle\xi^\epsilon,\textbf{t}^{\mu  }_{\mu^\epsilon}-\textbf{i}\rangle d\mu^\epsilon+\int_{\Omega(t)}\langle\xi^\epsilon,\textbf{t}^{\tilde{\mu}}_{\mu^\epsilon}-\textbf{t}^{{\mu}}_{\mu^\epsilon}\rangle d\mu^\epsilon+\frac{\tilde{\lambda}}{2}(d_W(\mu,\mu^\epsilon)+C\epsilon^\frac{1}{d+2})^2.   
\end{equation*}
Let $\gamma^\epsilon$ be an optimal transport plan between $\mu, \mu^\epsilon$. The above
\begin{equation}\label{phiepsilon}
 \geq \int_{\Omega(t)^2}\langle\xi^\epsilon(y),x-y\rangle d\gamma^\epsilon+\int_{\Omega(t)}\langle\xi^\epsilon,\textbf{t}^{\tilde{\mu}}_{\mu^\epsilon}-\textbf{t}^{{\mu}}_{\mu^\epsilon}\rangle d\mu^\epsilon-Cd^2_W(\mu,\mu^\epsilon)-C\epsilon^\frac{2}{d+2}.   
\end{equation}
Take $w_2=\mu^\epsilon$ in \eqref{phixi},
\begin{equation}\label{phiineq}
    \phi\left(\mu^\epsilon(t)\right)-\phi\left(\mu(t)\right)\geq \int_{\overline{\Omega(t)}}\langle\xi(x),y-x\rangle d\gamma^\epsilon+\frac{\tilde{\lambda}}{2}d_W^2(\mu,\mu^\epsilon).
\end{equation}

Next by H\"{o}lder's inequality
\[\left|\int_{\Omega_T}\langle\xi^\epsilon,\textbf{t}^{\tilde{\mu}}_{\mu^\epsilon}-\textbf{t}^{{\mu}}_{\mu^\epsilon}\rangle d\mu^\epsilon dt \right|\leq \left(\int_{\Omega_T}|\xi^\epsilon|^2 d\mu^\epsilon dt\right)^\frac{1}{2}\left( \int_{\Omega_T}|\textbf{t}^{\tilde{\mu}}_{\mu^\epsilon}-\textbf{t}^{{\mu}}_{\mu^\epsilon}|^2 d\mu^\epsilon dt\right)^\frac{1}{2}.\]
By Lemma \ref{slope}, $\int_{\Omega_T}|\xi^\epsilon|^2 d\mu^\epsilon dt$ is uniformly bounded and 
\[\left(\int_{\Omega(t)}|\textbf{t}^{\tilde{\mu}}_{\mu^\epsilon}-\textbf{t}^{{\mu}}_{\mu^\epsilon}|^2 d\mu^\epsilon\right)^\frac{1}{2}=d_{\mu^\epsilon}(\mu,\tilde{\mu})\] is the Pseudo-Wasserstein distance induced by $\mu^\epsilon\in \mathcal{P}_2^a$.
So
\[\left|\int_{\Omega_T}\langle\xi^\epsilon,\textbf{t}^{\tilde{\mu}}_{\mu^\epsilon}-\textbf{t}^{{\mu}}_{\mu^\epsilon}\rangle d\mu^\epsilon dt \right|\leq C\int_0^Td_{\mu^\epsilon}(\mu,\tilde{\mu})dt\leq C\epsilon^{\frac{1}{d+2}}T.\]
This inequality as well as \eqref{phiepsilon} \eqref{phiineq} gives for any $T'\in[0,T]$
\[
\int_{\overline{\Omega_{T'}}^2}\langle-\xi(x)+\xi^\epsilon(y),x-y\rangle d\gamma^\epsilon dt\leq \]\[ \int_0^{T'}(\mathcal{U}^\epsilon(\tilde{\mu}(t))-\mathcal{U}^\epsilon({\mu}^\epsilon(t)))dt+C\int_0^{T'}d_W^2(\mu,\mu^\epsilon)dt+C\epsilon^{\frac{1}{d+2}}T'.
\]
Because $\tilde{u}(x,t)\leq \epsilon^{-\frac{d}{d+2}}$ pointwise and the domain is bounded, we have
\[\int_0^{T'}\mathcal{U}^\epsilon(\tilde{\mu})dt=\epsilon\int_{\Omega_{T'}}(\tilde{u}\log \tilde{u})(x,t)dxdt\leq C\epsilon^{\frac{1}{d+2}}T'.\] 
Also note $(u^\epsilon\log u^\epsilon)$ is bounded below, we have $-\mathcal{U}^\epsilon(\mu^\epsilon(t))\leq C\epsilon.$ Then
\begin{equation}
\label{ineqxi}
\int_{\overline{\Omega_{T'}}^2}\langle-\xi(x)+\xi^\epsilon(y),x-y\rangle d\gamma^\epsilon dt\leq  C\int_0^{T'}d_W^2(\mu,\mu^\epsilon)dt+C\epsilon^{\frac{1}{d+2}}T'.
\end{equation}

By Theorem 8.4.7 and Lemma 4.3.4 from \cite{gflow}, we find
\begin{equation*}
\frac{d}{d t}d_W^2(\mu,\mu^\epsilon)\leq 2\int_{\overline{\Omega(t)}^2}\langle v(x)-v^\epsilon(y),x-y\rangle d\gamma^\epsilon
=2\int_{\overline{\Omega(t)}^2}\langle \xi^\epsilon(y)-\xi(x),x-y\rangle d\gamma^\epsilon
.\end{equation*}
By \eqref{ineqxi} and $d_W^2(\mu,\mu^\epsilon)(0)=0$, we deduce that
\[ d_W^2(\mu,\mu^\epsilon)(T')\leq C \int_{0}^{T'}d_W^2(\mu,\mu^\epsilon)dt+\delta(\epsilon)T'\]
for all $T'\in[0, T]$ and $\delta(\epsilon)=C\epsilon^\frac{1}{d+2}$ for some constant $C$ depends only on the domain and universal constants. Then Gronwall's inequality finishes the proof that we have
\[d^2_W(\mu,\mu^\epsilon)(t)\leq \delta{(\epsilon)}te^{ct}.\]

Actually if we keep track of the constants,  $\delta(\epsilon)\leq C\epsilon^\beta$ for all $\beta\in(0,\frac{1}{d+1})$ where $C$ depends on $\beta,\lambda$, the volumes and diameters of $\Omega(t),t\in[0,T]$. 
\end{proof}

\medskip

\appendix

\section{Remark on the Modification Lemma}\label{counterexample}
We make a remark that the modification of $\mu$ done in Lemma \ref{modi} can not be replaced by simply convoluting $\mu$ with a smooth, positive, compactly supported function. We want to show that, the difference between one measure and a ``small perturbation" (including convolutions) of it can be large in the Pseudo-Wasserstein metric for some base measure.
To illustrate the main idea, let us consider the following base measure $v$ which is a sum of delta masses. And instead of convolution, we first consider small shifts.

Suppose in $\mathbb{R}^2$, $\epsilon>0$,
\[v=\frac{1}{2}\delta_{(-1,0)}+\frac{1}{2}\delta_{(1,0)},\; \mu_1=\frac{1}{2}\delta_{(-\epsilon,1)}+\frac{1}{2}\delta_{(\epsilon,-1)},\;
 \mu_2=\frac{1}{2}\delta_{(\epsilon,1)}+\frac{1}{2}\delta_{(-\epsilon,-1)}.
\]
Then the optimal transport maps from $v$ to $\mu_i$ are
\begin{equation*}
    \textbf{t}_v^{\mu_1}(x)=\begin{cases}(\epsilon,-1) &\text{ when }x=(1,0),\\
(-\epsilon,1) &\text{ when }x=(-1,0);
\end{cases}
\quad
\textbf{t}_v^{\mu_2}(x)=\begin{cases}(\epsilon,1) &\text{ when }x=(1,0),\\
(\epsilon,-1) &\text{ when }x=(-1,0).
\end{cases}
\end{equation*}
So
\[d^2_v(\mu_1,\mu_2)=\int_{\mathbb{R}^2} |\textbf{t}_v^{\mu_1}-\textbf{t}_v^{\mu_2}|^2dv=4.\]
For small $\epsilon$, geometrically $\mu_2$ is just a small perturbation of $\mu_1$. This shows that a little shift may cause a large difference in Pseudo-Wasserstein metric. And so it is possible that the convolution of $\mu$ with $\frac{1}{\epsilon^d}\varphi(\frac{\cdot}{\epsilon})$ ($\varphi$ is a bump function and $\epsilon$ is a small positive value) is far away from $\mu$ in view of the Pseudo-Wasserstein metric.

\section{Proof of Proposition \ref{2.1}}
To solve this ODE, we cite the following result from \cite{BV} about the existence of differential inclusions.
\begin{theorem}(Theorem 5.1 \cite{BV}) \label{bv}
Assume $S(t)=\overline{\Omega(t)}^N \subset \mathbb{R}^{dN}$ satisfies the following: 
\begin{eqnarray*}
&&\text{for each }t\in [0,T], S(t) \text{ is nonempty and $r$-prox-regular;}\\
&&\text{the set $S(t)$ varies absolutely continuously (see (H3) \cite{BV}). }
\end{eqnarray*}
Also assume that $F:  \mathbb{R}^{dN} \times [0,T]\rightarrow \left\{\text{nonempty convex compact subsets of  } \mathbb{R}^{dN}\right\}$ satisfies:
\begin{eqnarray*}
&&F(x ,t ) \text{ is upper semicontinuous in $(x,t)$;}\\
&&\text{there exists }\beta(t)\in L^1([0,T],\mathbb{R})\text{ non-negative, such that } |F(x,t)|\leq \beta(t)(1+|x|).
\end{eqnarray*}
Then for any $x_0\in \overline{\Omega(0)}^N$, the following sweeping process with perturbation
\begin{equation*}
\left\{\begin{aligned}
&-\dot{x}(t)\in N\left(\overline{\Omega(t)}^N,x(t)\right)+F(x(t),t) \quad a.e.\quad t\in I,\\
&x(0)=x_0
\end{aligned}\right.
\end{equation*}
has at least one absolutely continuous (in supremum norm) solution $x(t)$. Here $N\left(\overline{\Omega(t)}^N,x(t)\right)$ denotes the normal cone at $x(t)$ if $x(t)$ is on the boundary, otherwise it is an empty set. 
\end{theorem}

In our case of $C^1$ boundary, the normal cone simply means the collection of all outer normal vectors. Then we prove the following proposition.

\begin{proof}(of Proposition \ref{2.1})
To apply Theorem \ref{bv}, we need to verify all the conditions. 
Proposition 2.5 in \cite{prox} showed that if $\Omega(t)$ is $r$-prox-regular then so is $\Omega(t)^N$. The absolute continuity of $\Omega(t)^N$ follows from condition \textsc{(O1)}. Also the upper semi-continuity of $w(\mu)(x,t)$ follows from the definition \eqref{Fxmu}. We may write $w(x,t)$ for abbreviation.
For each $N$ and all $t$, the linear growth of $|w(\cdot,t)|$ can be proved by definition as well as estimate \eqref{boundm2}. 

In all, the assumptions in Theorem \ref{bv} are satisfied, and thus there exists $x(t)=(x_1,...,x_N)(t)$ absolutely continuous such that
\begin{equation*}
\left\{\begin{aligned}
 &-\dot{x}(t)\in N\left(\overline{\Omega(t)}^N,x(t)\right)-\left(w(x_1,t),...,w(x_N,t)\right),\quad  a.e. \text{ }t\geq 0,
 \\
 &x(t)\in \overline{\Omega(t)}^N,\quad x(0)=(x_{1,0},x_{2,0}...x_{N,0}).
\end{aligned}
\right.
\end{equation*}

Then we show the solutions above are the solutions for the projected systems similarly as did in Lemma 2.4 \cite{prox}.
Write $-\dot{x_i}(t)=h_i(x,t)-w(x_i,t)$. Note for $x\in \overline{\Omega(t)}^N$,
\[ h(x,t)=(h_1,...,h_N)(x,t)\in N(\overline{\Omega(t)}^N,x(t))\iff h_i(x,t)\in N(\overline{\Omega(t)},x_i(t)).\]
So we can write $k_i(x,t) n(x_i,t)=h_i(x,t)$ such that
\begin{equation*}\label{defki}
\dot{x_i}(t)=w(x_i,t)-k_i(x,t)  n(x_i,t).
\end{equation*}
Set $k_i(x,t)=0$ if $x_i(t) \notin \partial \Omega(t)$ and we have $k_i\geq 0$. Claim
\[w(x_i,t)-k_i(x,t)  n(x_i,t)= P_{x,t} \left(w(x_i,t)\right) \text{ a.e. in time}.\]
Since $x_i(\cdot)$ is absolutely continuous in $\mathbb{R}$, we only need to consider all the $t\in[0,T]$ where $x_i(t)$ are differentiable. Also we only need to consider the case when $x_i(t)\in \partial \Omega(t)$.
Because $x_i(t)$ is supported in $\overline{\Omega(t)}$, $\dot{x_i}(t)\cdot n(x_i,t)\leq c(x_i,t)$.
If $\dot{x_i}\cdot n=c$, we have $w(x_i,t)\cdot n(x_i,t)\geq c(x_i,t)$. By definition of $ P_{x,t} $, we only need to check the equality in the normal direction that
\[ P_{x,t} \left(w(x_i,t)\right)\cdot n(x_i,t)=c(x_i,t)=\left(w(x_i,t)-k_i(x,t)n(x_i,t)\right)\cdot n(x_i,t).\]
If $\dot{x_i}(t)\cdot n(x_i,t)<c(x_i,t),$
in view of the continuity of $c(x,t)$, $x_i(t')$ is in the interior of $\Omega(t)$ a.e. for $t'$ close to $t$. Then by continuity, $k_i(x,t)=0$. We also have at $t$
\[ P_{x,t}\left(w(x_i,t)\right)=w(x_i,t)=w(x_i,t)-k_i(x,{t})n(x_i,{t}).
\]
So \eqref{ODE} is satisfied a.e. for all $t$. 
\end{proof}

\section{Proof of Theorem \ref{uniqueness}}
\begin{proof}
Let $\mu$ be a solution to equation \eqref{eqn2'} with initial data $\mu_0$. First we show that $u(t)$ is bounded in $L^2$ norm. Let $\eta_\iota(x)$ be a space mollifier: a non-negative function supported in $B_\iota (\iota<<1)$ with total mass $1$.
Set 
\[u^\iota(x)=\int_{\Omega}\eta_\iota(x-y) u(y)dy=:\eta_\iota*u(x).\] 
Since $u(t)\in L^{1}(\Omega)$, we have for each $t$, $u^\iota(t) \rightarrow u(t)$ in $L^{1}(\Omega)$. Fix any function $\varphi\in C_c^\infty\left((0,T)\right)$ and $x$, then $\eta_\iota(x-\cdot)\varphi(t)$ is a smooth function. From the definition of weak solutions
\[\int_0^T u^\iota (x,t)\varphi'(t) dt=\int_0^T \varphi(t)\int_{\Omega}\left(\nabla u+\nabla V u+(\nabla W*\mu)u\right)(y,t)\cdot\nabla_y\eta_\iota(x-y)dydt.\]
This implies that for all $x$ and $\iota$
\[u^\iota(x,\cdot)\in W^{1,\infty}([0,T],dt)\quad\text{ and }\]
\[\frac{\partial}{\partial t}u^\iota(x,t)=\int_{\Omega}\left(\nabla u+\nabla V u+(\nabla W*\mu)u\right)*\nabla\eta_\iota dy \text{  a.e. }dt.\]
By multiplying $u^\iota$ on both sides and integrating over $\Omega$, we deduce for a.e. $dt$
\begin{align*}
    \frac{1}{2}\frac{d}{d t}\|u^\iota\|^2_{L^2(\Omega)}&=\int_{{\Omega}^{2}}\left(\nabla u+\nabla V u+(\nabla W*\mu)u\right)(y)\cdot\left(\nabla_x\eta_\iota(x-y)\right)u^\iota (x)dydx\\
    &=-\int_{{\Omega}}\left(\nabla u+\nabla V u+(\nabla W*\mu)u\right)(y)\cdot\left(\nabla\eta_\iota *(\eta_\iota*u)\right) (y)dy\\
    &=-\int_{\Omega}\left(\nabla u+\nabla V u+(\nabla W*\mu)u\right)*\eta_\iota\cdot\nabla u^\iota dx.
\end{align*}
Since $|D^2 V|,|D^2 W|$ are bounded in compact sets and $\eta_\iota$ is supported in a small ball, 
\[|(\nabla Vu)*\eta_\iota-\nabla V u^\iota|=O(\iota)u^\iota,\]
\[|\left((\nabla W*\mu) u\right)*\eta_\iota-(\nabla W*\mu) u^\iota|=O(\iota)u^\iota.\]
Then for a.e. $dt$:
\[\frac{1}{2}\frac{d}{d t}\|u^\iota\|^2_{L^2(\Omega)}\leq -\int_{\Omega} (|\nabla u^\iota|^2+u^\iota\nabla V\cdot \nabla u^\iota+u^\iota\nabla W*\mu\cdot \nabla u^\iota)dx+C\iota\int_{\Omega}u^\iota|\nabla u^\iota|dx\]
\[\leq-\|\nabla u^\iota\|^2_{L^2(\Omega)}+C\|u^\iota\|_{L^2(\Omega)}\|\nabla u^\iota\|_{L^2(\Omega)}
\leq C\|u^\iota\|^2_{L^2(\Omega)}.\]
In the above, we used the  boundedness of $\nabla V, \nabla W$ and mean inequality. 
Also notice that $\mu(\cdot)$ is an absolutely continuous curve in $\mathcal{P}^2_a(\Omega)$, so $u^\iota(t)dx \rightarrow u^\iota_0 dx$ in Wasserstein distance as $t\rightarrow 0$. 
By Gronwall inequality, we find out that for $t\leq T$, $u^\iota$ is uniformly bounded in $L^2(\mathbb{R}^d;dx)$. 
Then there's a subsequence of $u^\iota(\cdot)$ that converges to $w(\cdot)$ in $L^2(\Omega\times[0,T];dxdt)$ for some $w(\cdot)$. But since $u^\iota\rightarrow u$ in $L^{1}(\Omega\times[0,T])$, $u(\cdot)=w(\cdot)$ in $L^2(\Omega\times[0,T])$. By uniform boundedness of $u^\iota$ in $L^2(\Omega)$, $u(t)$ is uniformly bounded in $L^2(\Omega)$ for almost all $t\in[0,T]$.

Now suppose there are two solutions $\mu_i=u_i \mathcal{L}^d (i=1,2)$ to equation \eqref{eqn2'} with initial data $\mu^i_0$. 
Let $\rho=u_1-u_2$ and $ \rho^\iota=\rho*\eta_\iota$. We know
\[\rho^\iota_t=\left(\nabla \rho+\nabla V \rho+(\nabla W *\mu_1)u_1-(\nabla W *\mu_2)u_2\right)*\nabla \eta_\iota.\]
Thus
\[
\frac{1}{2}\frac{d}{d t}\|\rho^\iota\|^2_{L^2(\Omega)}=-\int_{\Omega} (|\nabla \rho^\iota|^2+\rho^\iota\nabla V\cdot \nabla\rho^\iota)dx
-\int_{\Omega^{2}}\left(\nabla W(x-y)\cdot \nabla\rho^\iota(x)\,\rho(y)u_1^\iota(x)\right.\]
\begin{equation}\label{rhol2}
\left.+\nabla W(x-y)\cdot\nabla\rho^\iota(x)\,u_2(y)\rho^\iota(x)\right)dydx+O(\iota)\int_{\Omega}\rho^\iota |\nabla \rho^\iota|+u_1^\iota |\nabla \rho^\iota| dx.\end{equation}
Notice
\begin{align*}
\left|\int_{\Omega}(\nabla W*\rho) \cdot \nabla\rho^\iota\, u_1^\iota dx\right|&\leq C\int_{\Omega(t)}|\rho|dx \int_{\Omega}|u_1^\iota \nabla\rho^\iota|dx\\
&\leq C\left\|\rho\right\|_{L^2\left(\Omega(t)\right)}\left\|u_1^\iota\right\|_{L^2(\Omega)}\left\|\nabla\rho^\iota\right\|_{L^2(\Omega)}\\
&\leq C\|\rho\|_{L^2\left(\Omega(t)\right)}\|\nabla\rho^\iota\|_{L^2(\Omega)},
\end{align*}
\[\text{ and  }\left|\int_{\Omega}(\nabla W*\mu_2) \cdot \nabla\rho^\iota\, \rho^\iota  dx\right|\leq C\|\rho^\iota\|_{L^2(\Omega)}\|\nabla\rho^\iota\|_{L^2(\Omega)}.\]
By \eqref{rhol2} and mean inequality, we deduce
\begin{equation*}
\frac{1}{2}\frac{d}{d t}\|\rho^\iota\|^2_{\Omega}
\leq C\|\rho^\iota\|^2_{L^2(\Omega)}+C\|\rho\|^2_{L^2\left(\Omega\right)}.
\end{equation*} 
Since $\rho^\iota(\cdot,t)\rightarrow \rho(\cdot,t)$ in $L^2(\Omega)$ a.e.$dt$, by Gronwall we have the stability result which also shows the uniqueness of solutions in $L^2$ norm.
\end{proof}

\section{Proof of Lemma \ref{slope}}
\begin{proof}
 
Recall the JKO scheme \eqref{jkoscheme}, let $\mu^\epsilon_\tau$ be defined as the discrete solution with time step $\tau$. From the above we know that for $t\leq T$, $\mu^\epsilon_\tau$ converges to $\mu^\epsilon$ uniformly in Wasserstein metric.
For abbreviation, write $\mu^{\epsilon,k}_\tau=\mu^\epsilon_\tau(k\tau)$. We want to show that $\mu^{\epsilon,k}_\tau$ has a finite metric slope at time $k\tau$. As did in Lemma 3.1.3 (Slope estimate) \cite{gflow}, let $v\in Dom(\phi^\epsilon,k\tau)$. Then
\[\phi^\epsilon(\mu^{\epsilon,k}_\tau)-\phi^\epsilon(v)\leq \frac{1}{2\tau}\left(d_W^2(v,\mu^{\epsilon,k-1}_\tau)-d_W^2(\mu^{\epsilon,k-1}_\tau,\mu^{\epsilon,k}_\tau)\right)\]
\[\leq \frac{1}{2\tau}d_W(v,\mu^{\epsilon,k}_\tau)(d_W(v,\mu^{\epsilon,k-1}_\tau)+d_W\left(\mu^{\epsilon,k-1}_\tau,\mu^{\epsilon,k}_\tau)\right).\]
Then
\[|\partial \phi^\epsilon(k\tau)|(\mu^{\epsilon,k}_\tau)=\limsup_{v\rightarrow \mu^{\epsilon,k}_\tau}\frac{\left(\phi^\epsilon(\mu^{\epsilon,k}_\tau)-\phi^\epsilon(v)\right)^+}{d_W(v,\mu^{\epsilon,k}_\tau)}\]\[\leq \limsup_{v\rightarrow \mu^{\epsilon,k}}\frac{1}{2\tau}(d_W(v,\mu^{\epsilon,k-1}_\tau)+d_W\left(\mu^{\epsilon,k-1}_\tau,\mu^{\epsilon,k}_\tau)\right)=\frac{d_W(\mu^{\epsilon,k-1}_\tau,\mu^{\epsilon,k}_\tau)}{\tau}.\]

Next by Proposition \ref{prop 2},
\begin{equation}\label{metricslope}
\tau\sum_{1\leq k\leq T/\tau} |\partial\phi^{\epsilon,k}(k\tau)|(\mu_\tau^{\epsilon,k})\leq \left(\sum_{1\leq k\leq T/\tau}\frac{d_W^2(\mu^{\epsilon,k}_\tau,\mu^{\epsilon,k-1}_\tau)}{\tau}\right)^\frac{1}{2}\leq C.
\end{equation}
The constant $C$ above is independent of $\tau,\epsilon$.

Let $\textbf{t}\in L^2(d\mu^{\epsilon,k}_\tau,\mathbb{R}^d)$ be a transport map which is $\mu^{\epsilon,k}_\tau-$a.e. differentiable and compactly supported in $\Omega(k\tau)$. 
Write $\textbf{t}_s=(1-s)\textbf{i}+s\textbf{t}$.
For $s$ small, 
\[\phi^\epsilon(\mu^{\epsilon,k}_\tau)-
\phi^\epsilon(\textbf{t}_s\#\mu^{\epsilon,k}_\tau)=\int_{\Omega(k\tau)}-u^{\epsilon,k}_\tau\log (\det D\textbf{t}_s)dx+\int_{\Omega(k\tau)}\langle \nabla V+\nabla W*\mu^{\epsilon,k}_\tau,\textbf{t}_s-\textbf{i} \rangle d \mu^{\epsilon,k}_\tau\]\[+C\left(\|D^2 V\|_{L^\infty}+\|D^2 W\|_{L^\infty}\right)\left(\int_{\Omega(k\tau)}|\textbf{t}_s-\textbf{i}|^2 d \mu^{\epsilon,k}_\tau\right).\]
Due to the expansion $\det D \textbf{t}_s=1+s \tr\nabla (\textbf{t}-\textbf{i})+o(s)$ as well as the compact support of $\textbf{t}-\textbf{i}$ in $\Omega(k\tau)$, we find the above
\begin{equation*}\label{direct}
    =\int_{\Omega(k\tau)}\langle v^\epsilon_\tau, \textbf{t}_s-\textbf{i}\rangle d \mu^\epsilon_\tau+o\left(d_W(\mu^\epsilon_\tau,\textbf{t}_s\#\mu^\epsilon_\tau)\right).
\end{equation*}
Here $v^{\epsilon,k}_\tau:=\epsilon\frac{\nabla u^{\epsilon,k}_\tau}{u^{\epsilon,k}_\tau}+\nabla V+\nabla W*\mu^{\epsilon,k}_\tau$. Note $d_W(\mu^{\epsilon,k}_\tau,\textbf{t}_s\#\mu^{\epsilon,k}_\tau)=s d_W(\mu^{\epsilon,k}_\tau,\textbf{t}\mu^{\epsilon,k}_\tau)$. We divide the above by $s$ and let $s\rightarrow 0$ to obtain
\[
\int_{\Omega(k\tau)}\langle v^{\epsilon,k}_\tau, \textbf{t}\rangle d \mu^\epsilon_\tau\leq |\partial \phi^\epsilon(k\tau)|(u^{\epsilon,k}_\tau)\|\textbf{t}\|_{L^2(d\mu_\tau^{\epsilon,k})}
\]
Since $\textbf{t}$ can be any compactly supported vector field, we have
\[\left(\int_{\Omega(k\tau)}|v_\tau^{\epsilon,k}|^2d\mu^{\epsilon,k}_\tau\right) \leq  C|\partial \phi^\epsilon(k\tau)|(u^{\epsilon,k}_\tau).\]
Let us define
\[v^\epsilon_\tau(t):=v^{\epsilon,k}_\tau \quad\text{if }t\in ((k-1)\tau,k\tau].\]
Then by \eqref{jkoscheme}, \eqref{metricslope}, $v^\epsilon_\tau\in L^2(u^\epsilon_\tau dxdt, \mathbb{R}^d)$ with a bound independent of both $\tau$ and $\epsilon$. Recall $\mu^\epsilon_\tau\rightarrow \mu^\epsilon$ in Wasserstein metric uniformly for all $t\in[0,T]$, then by Theorem 5.4.4 \cite{gflow} $v^\epsilon_\tau$ converges weakly to some $w\in L^2(u^\epsilon dxdt, \mathbb{R}^d)$. From the previous discussion and the equation, we know that $v^\epsilon_\tau$ converges weakly to $v^\epsilon=\epsilon\frac{\nabla u^\epsilon}{u^\epsilon}+\nabla V+\nabla W*\mu^\epsilon$. 
It is not hard to see that such limit is unique, so we have $v^\epsilon=w$ a.e. $d\mu^\epsilon dt$. This finishes the proof with $C$ independent of $\epsilon$
\end{proof}


\begin{thebibliography}{10}

\bibitem{gflow}
Luigi Ambrosio, Nicola Gigli, and Giuseppe Savar{\'e}.
\newblock {\em Gradient flows: in metric spaces and in the space of probability
  measures}.
\newblock Springer Science \& Business Media, 2008.

\bibitem{carrillo}
Jos{\'e}~Antonio Carrillo, Marco DiFrancesco, Alessio Figalli, Thomas Laurent,
  Dejan Slep{\v{c}}ev, et~al.
\newblock Global-in-time weak measure solutions and finite-time aggregation for
  nonlocal interaction equations.
\newblock {\em Duke Mathematical Journal}, 156(2):229--271, 2011.

\bibitem{6}
Jos{\'e}~Antonio Carrillo, Stefano Lisini, and Edoardo Mainini.
\newblock Gradient flows for non-smooth interaction potentials.
\newblock {\em Nonlinear Analysis: Theory, Methods \& Applications},
  100:122--147, 2014.

\bibitem{prox}
Jos{\'e}~Antonio Carrillo, Dejan Slepcev, and Lijiang Wu.
\newblock Nonlocal-interaction equations on uniformly prox-regular sets.
\newblock {\em Discrete and Continuous Dynamical Systems-Series A},
  36(3):1209--1247, 2014.

\bibitem{proxregular}
Francis~H Clarke, RJ~Stern, and PR~Wolenski.
\newblock Proximal smoothness and the lower-c2 property.
\newblock {\em J. Convex Anal}, 2(1/2):117--144, 1995.

\bibitem{cozzi2016aggregation}
Elaine Cozzi, Gung-Min Gie, and James~P Kelliher.
\newblock The aggregation equation with newtonian potential: The vanishing
  viscosity limit.
\newblock {\em Journal of Mathematical Analysis and Applications},
  453(2):841--893, 2017.

\bibitem{katy}
Katy Craig.
\newblock The exponential formula for the wasserstein metric.
\newblock {\em ESAIM: Control, Optimisation and Calculus of Variations},
  22(1):169--187, 2016.

\bibitem{sobolev1}
Manuel Del~Pino and Jean Dolbeault.
\newblock The optimal euclidean l p-sobolev logarithmic inequality.
\newblock {\em Journal of Functional Analysis}, 197(1):151--161, 2003.

\bibitem{sweeping}
Simone Di~Marino, Bertrand Maury, and Filippo Santambrogio.
\newblock Measure sweeping processes.
\newblock {\em Journal of Convex Analysis}, 23(2):567--601, 2016.

\bibitem{unique}
Simone Di~Marino and Alp{\'a}r~Rich{\'a}rd M{\'e}sz{\'a}ros.
\newblock Uniqueness issues for evolution equations with density constraints.
\newblock {\em Mathematical Models and Methods in Applied Sciences},
  26(09):1761--1783, 2016.

\bibitem{BV}
Jean~Fenel Edmond and Lionel Thibault.
\newblock Bv solutions of nonconvex sweeping process differential inclusion
  with perturbation.
\newblock {\em Journal of Differential Equations}, 226(1):135--179, 2006.

\bibitem{10}
Irene Fonseca and Giovanni Leoni.
\newblock {\em Modern Methods in the Calculus of Variations: L\^{} p Spaces}.
\newblock Springer Science \& Business Media, 2007.

\bibitem{sobolev}
Leonard Gross.
\newblock Logarithmic sobolev inequalities.
\newblock {\em American Journal of Mathematics}, 97(4):1061--1083, 1975.

\bibitem{JKO}
Richard Jordan, David Kinderlehrer, and Felix Otto.
\newblock The variational formulation of the fokker--planck equation.
\newblock {\em SIAM journal on mathematical analysis}, 29(1):1--17, 1998.

\bibitem{mccann1995}
Robert~J McCann et~al.
\newblock Existence and uniqueness of monotone measure-preserving maps.
\newblock {\em Duke Mathematical Journal}, 80(2):309--324, 1995.

\bibitem{la}
Luca Petrelli and Adrian Tudorascu.
\newblock Variational principle for general diffusion problems.
\newblock {\em Applied Mathematics and Optimization}, 50(3):229--257, 2004.

\bibitem{wuli}
Lijiang Wu and Dejan Slep{\v{c}}ev.
\newblock Nonlocal interaction equations in environments with heterogeneities
  and boundaries.
\newblock {\em Communications in Partial Differential Equations},
  40(7):1241--1281, 2015.

\end{thebibliography}

\end{document}